\documentclass[11pt]{amsart}
\usepackage{amsmath, amscd, amssymb}
\usepackage[enableskew,vcentermath]{youngtab}
\usepackage{ytableau}
\usepackage[mathscr]{eucal}
\usepackage[normalem]{ulem} 
\usepackage{soul} 
\usepackage[all,cmtip]{xy}
\usepackage{tikz-cd}
\usepackage{color} 
\setulcolor{red}
\sethlcolor{yellow}
\setstcolor{blue}
\usepackage{amscd}
\usepackage{xfrac} 

\usepackage{geometry}
\usepackage{amsfonts}
\usepackage[colorlinks,linkcolor=blue,citecolor=blue, pdfstartview=FitH]{hyperref}
\usepackage{graphicx}
\usepackage{cite}

\numberwithin{equation}{section}

\theoremstyle{definition}
\newtheorem{thm}{Theorem}[section]
\newtheorem{theorem}[thm]{Theorem}

\newtheorem{lemma}[thm]{Lemma}
\newtheorem{corollary}[thm]{Corollary}
\newtheorem{proposition}[thm]{Proposition}

\newtheorem{remark}[thm]{Remark}

\newtheorem{definition}[thm]{Definition}

\newtheorem{example}[thm]{Example}

\newtheorem{defn-thm}[thm]{Definition-Theorem}
\newtheorem{conjecture}[thm]{Conjecture}

\newenvironment{observe}{\noindent\textcolor{blue}{\textit{Observation}}.}{\hfill \textcolor{blue}{$\blacktriangleleft$}\par}
\newtheorem*{theorem*}{Theorem}
\newtheorem*{proposition*}{Proposition}
\usepackage{todonotes}
\definecolor{pistachio}{rgb}{0.58, 0.77, 0.45}
\definecolor{eggshell}{rgb}{0.94, 0.92, 0.84}

\newcommand{\sE}{{\mathcal E}}

\newcommand{\sK}{{\mathcal K}}

\newcommand{\sN}{{\mathcal N}}
\newcommand{\sO}{{\mathcal O}}

\newcommand{\sT}{{\mathcal T}}

\newcommand{\sX}{{\mathcal X}}

\newcommand{\cE}{{\mathcal E}}

\newcommand{\cN}{{\mathcal N}}
\newcommand{\cO}{{\mathcal O}}
\newcommand{\cP}{{\mathcal P}}

\newcommand{\g}{{\mathfrak g}}

\newcommand{\gt}{\mathfrak{t}}

\newcommand{\bW}{\mathsf{W}}

\newcommand{\GL}{\operatorname{GL}}
\newcommand{\SL}{\operatorname{SL}}
\newcommand{\Sp}{\operatorname{Sp}}

\newcommand{\SO}{\operatorname{SO}}

\newcommand{\Ker}{\operatorname{Ker}}

\renewcommand{\Im}{\operatorname{Im}}

\newcommand{\Tr}{\operatorname{Tr}}

\newcommand{\Spec}{\operatorname{Spec}}

\newcommand{\RT}{\operatorname{RT}}

\newcommand{\Gal}{\operatorname{Gal}}

\newcommand{\Gr}{\operatorname{Gr}}

\newcommand{\codim}{{\operatorname{codim}}}

\newcommand{\gr}{{\operatorname{gr}}}

\newcommand{\End}{\operatorname{End}}
\newcommand{\Ad}{\operatorname{Ad}}

\newcommand{\btheorem}{\begin{theorem}}
\newcommand{\etheorem}{\end{theorem}}
\newcommand{\bproposition}{\begin{proposition}}
\newcommand{\eproposition}{\end{proposition}}
\newcommand{\bdefinition}{\begin{definition}}
\newcommand{\edefinition}{\end{definition}}
\newcommand{\bcorollary}{\begin{corollary}}
\newcommand{\ecorollary}{\end{corollary}}
\newcommand{\bproof}{\begin{proof}}
\newcommand{\eproof}{\end{proof}}
\newcommand{\bremark}{\begin{remark}}
\newcommand{\eremark}{\end{remark}}
\newcommand{\eexample}{\end{example}}
\newcommand{\bexample}{\begin{example}}
\newcommand{\elemma}{\end{lemma}}
\newcommand{\blemma}{\begin{lemma}}
\newcommand{\bobserve}{\begin{observe}}
\newcommand{\eobserve}{\end{observe}}

\newcommand{\ord}{\tx{ord}}
\renewcommand{\bar}{\overline}

\renewcommand{\phi}{\varphi}
\newcommand{\ee}{\end{eqnarray*}}
\newcommand{\be}{\begin{eqnarray*}}

\newcommand{\beq}{\begin{equation}}
\newcommand{\eeq}{\end{equation}}
\newcommand{\bd}{\begin{enumerate}}
\newcommand{\ed}{\end{enumerate}}
\newcommand{\bti}{\begin{tikzcd}}
\newcommand{\eti}{\end{tikzcd}}

\renewcommand{\tilde}{\widetilde}

\newcommand{\NP}{\operatorname{NP}}

\renewcommand{\bf}[1]{\mathbf{#1}}
\newcommand{\tx}[1]{\text{#1}}

\newcommand{\rk}{\operatorname{rk}}

\title[Minimal reduction type in classical cases]{Minimal reduction type in classical cases}

\author{Bin Wang}
\address{Hetao Institute of Mathematics and Interdisciplinary Sciences, Shenzhen, China}
\email{matph92@gmail.com}

\author{Xueqing Wen}
\address{Chongqing University of Technology, No. 69, Hongguang Avenue, Banan District, Chongqing, 400054, China.}
\email{wenxq@cqut.edu.cn}

\author{Yaoxiong Wen}
\address{Zhejiang University, 866 Yuhangtang Rd, Hangzhou, 310058, P. R. China}
\email{y.x.wen.math@gmail.com}

\date{}

\begin{document}

\begin{abstract}
    We prove Yun’s minimal reduction conjecture for all classical groups. More precisely, for any topologically nilpotent regular semisimple element $\gamma$, we show that the associated minimal reduction set $\RT_{\min}(\gamma)$ consists of a single nilpotent orbit. This result confirms and extends Yun’s earlier work in types~A and~C, and resolves the remaining cases in types~B and~D. Moreover, we provide an explicit and effective procedure for determining $\RT_{\min}(\gamma)$.
\end{abstract}

\maketitle
\tableofcontents

\section{Introduction}

Let $G$ be a connected reductive group over $\mathbb{C}$ with Lie algebra $\g$. Fix a Borel subgroup $B \subset G$ containing a maximal torus $T$, and let $\bW = \bW(G, T)$ be the Weyl group, and let $\underline{\bW}$ denote the set of its conjugacy classes. Let $\cN$ denote the set of nilpotent elements in $\g$, and $\underline{\cN}$ the set of nilpotent orbits. A typical nilpotent orbit is denoted by $\bf{O}$.

\subsection{Kazhdan--Lusztig map and the minimal reduction map}

Let $L\g=\g(\!(t)\!)$ be the loop algebra of $\g$, and let $L^+\g=\g[\![t]\!]$ be its integral part. We denote by $L^{\heartsuit}\g \subset L\g$ the set of topologically nilpotent regular semisimple elements.

\subsubsection{Kazhdan--Lusztig map}

In \cite{KL88}, Kazhdan and Lusztig introduced a map
\[
    \mathrm{KL}: \underline{\cN} \longrightarrow \underline{\bW},
\]
defined as follows. 

Given a nilpotent orbit $\bf{O}$ and a representative $\gamma_0 \in \bf{O}$, choose a generic lift $\gamma \in \gamma_0 + tL^+\g$ such that $\gamma \in L^{\heartsuit}\g$. 
The centralizer $LG_\gamma$ then forms a maximal torus in the loop group $LG$. 
Kazhdan and Lusztig showed that the rational conjugacy class of $LG_\gamma$ is independent of the choice of generic lift. 
Moreover, rational conjugacy classes of maximal tori in $LG$ are in bijection with conjugacy classes in the Weyl group $\bW$, yielding the map $\mathrm{KL}$.

It has long been conjectured that the Kazhdan–Lusztig map is injective. This map was computed by Spaltenstein for classical groups \cite{Spa88,Spa90a} and for most exceptional groups \cite{Spa90b}. 

More recently, Yun \cite{Yun21} established the injectivity of the Kazhdan–Lusztig map via a different approach, by constructing the \emph{minimal reduction map}, which is a section of the Kazhdan–Lusztig map, thereby confirming Spaltenstein’s predictions \cite{Spa90b} for the remaining exceptional cases.

\subsubsection{Minimal reduction map}

The \emph{minimal reduction map}
\[
    \RT_{\min}: \underline{\bW} \longrightarrow \underline{\cN},
\]
defined using affine Springer fibers.

Let $\Gr := LG/L^+G$ be the affine Grassmannian, which is an ind-projective ind-scheme over $\mathbb{C}$. For $\gamma \in L^{\heartsuit}\g$, the associated affine Springer fiber is
\[
    \Gr_{\gamma}:=\{gL^+G\in \Gr \mid g^{-1}\gamma g\in L^+\g\}.
\]
There is a natural evaluation morphism
\begin{align*}
    ev_{\gamma}: \Gr_\gamma &\longrightarrow \underline{\cN}, \\
    gL^+G &\longmapsto g^{-1}\gamma g \pmod t .
\end{align*}
For a nilpotent orbit $\bf{O}$, let $\Gr_{\gamma,\bf{O}}$ denote the fiber of $ev_{\gamma}$ over $\bf{O}$, which is a locally closed subscheme of $\Gr_\gamma$.

\begin{definition}[Minimal reduction]
    Let $\RT(\gamma)$ be the set of nilpotent orbits $\bf{O}$ such that $\Gr_{\gamma,\bf{O}}$ is non-empty. We define $\RT_{\min}(\gamma)$ to be the set of minimal elements of $\RT(\gamma)$ with respect to the closure order.
\end{definition}

By \cite[\S4, Corollary~1]{KL88}, the regular nilpotent orbit always lies in $\RT(\gamma)$, hence $\RT_{\min}(\gamma)$ is non-empty. 
However, since the closure order on nilpotent orbits is only a partial order, the minimal reduction need not be unique a priori.

For a conjugacy class $[w]\in \underline{\bW}$, let $(L^{\heartsuit}\g)_{[w]}$ denote the subset of elements of type $[w]$ under the Kazhdan--Lusztig correspondence. Yun showed that there exists a Zariski open subset of \emph{shallow elements}, denoted by $(L^{\heartsuit}\g)^{sh}_{[w]}$, such that for any $\gamma$ in this subset, $\RT_{\min}(\gamma)$ always consists of a unique common nilpotent orbit. This orbit, denoted by $\RT_{\min}([w])$, defines a map
\[
    \RT_{\min}: \underline{\bW} \longrightarrow \underline{\cN},
\]
which Yun proved to be a section of the Kazhdan--Lusztig map. 

Meanwhile, Yun~\cite{Yun21} also shows that the minimal reduction map coincides with a map defined by Lusztig in \cite{Lus11}, whose definition does not involve affine Springer fiber. Most recently, Shan, Yan and Zhao~\cite{SYZ25} defined two cyclotomic level maps which are compatible with Lusztig's map and hence also the minimal reduction map.

For more general elements, Yun proposed the following conjecture.

\begin{conjecture}[Yun {\cite[Conjecture~1.15]{Yun21}}]\label{conj of Yun}
    For any $\gamma\in L^{\heartsuit}\g$, the minimal reduction set $\RT_{\min}(\gamma)$ consists of a single nilpotent orbit.
\end{conjecture}

Yun proved this conjecture for types~A,~C, and~$G_2$ in \cite{Yun21}, using the notion of skeletons introduced therein (see Definition~\ref{def:skeleton}). He also proves this for general reductive groups conditionally on the conjugacy classes of Weyl groups. In particular, he proves this for shallow elements for all classical groups and $G_2$.

\subsection{Main results}

In this paper, we confirm Yun’s conjecture for all classical groups, using a method different from that of \cite{Yun21}.

\begin{theorem}\label{thm.main}
    Let $G$ be a classical group. For any $\gamma \in L^\heartsuit \g$, the minimal reduction set $\RT_{\min}(\gamma)$ consists of a single nilpotent orbit. 
    Moreover, this orbit can be determined explicitly from the characteristic polynomial $\chi(\gamma)\in L^+\mathfrak{c}$.
\end{theorem}

We now briefly outline our construction. For the proof of uniqueness and existence, especially for the cases of type B and D, please refer to the beginning of Section~\ref{type BD}, where we also explain our strategy more explicitly. We first introduce the following notion.

\begin{definition}\label{def:balanced}
    Given positive integers $(n,m)$, write $n=mk+l$ with $0\leq l<m$. 
    The \emph{$m$-balanced partition} of $n$ is defined to be
    \[
        [(k+1)^{l},\, k^{m-l}],
    \]
    where the notation $a^b$ means that the part $a$ appears with multiplicity $b$.
\end{definition}

For any $\gamma\in L^\heartsuit \g$, we first decompose its characteristic polynomial $\chi(\gamma)$ according to its Newton polygon. 
For each factor $f_i$, let $n_i$ be its degree and $m_i$ the order of its constant term. 
We associate to $f_i$ the $m_i$-balanced partition of $n_i$. 
The collection of these balanced partitions is denoted by $\RT^{\mathrm{A}}_{\min}(\gamma)$.

For groups of type~A and~C, the partition $\RT^{\mathrm{A}}_{\min}(\gamma)$ coincides with the minimal reduction of $\gamma$. In particular, when $\gamma\in L^\heartsuit \mathfrak{sp}_{2n}$, the polynomial $\chi(\gamma)$ is self-dual, and hence $\RT^{\mathrm{A}}_{\min}(\gamma)$ automatically defines a partition of type~C.

For groups of type~B or~D, the situation is more subtle. 
In general, $\RT^{\mathrm{A}}_{\min}(\gamma)$ does not define a partition of type~B or~D. 
Moreover, among all partitions that are larger than $\RT^{\mathrm{A}}_{\min}(\gamma)$ for type~B or~D, minimal elements need not be unique \emph{a priori}.

Following a result of Spaltenstein (see Proposition~\ref{relation for type BD1*}), we introduce the notion of an \emph{admissible partition} associated with $\gamma$. 
We then prove that among all admissible partitions of type~B or~D that dominate $\RT^{\mathrm{A}}_{\min}(\gamma)$, there exists a unique minimal one. 
This partition is shown to be precisely the partition of the minimal reduction of $\gamma$.

The uniqueness of minimal reductions has several important consequences.

Affine Springer fibers have played a central role in representation theory and geometry since the seminal work of Kazhdan--Lusztig~\cite{KL88}. 
Bezrukavnikov~\cite{Bez96} proved the dimension formula for affine Springer fibers—conjectured in~\cite{KL88}—by studying the dimension of $\Gr_{\gamma,\bf{O}^{\mathrm{reg}}}$. 
For various nilpotent orbits $\bf{O}$, the strata $\Gr_{\gamma,\bf{O}}$ form a stratification of the affine Springer fiber $\Gr_{\gamma}$, with $\Gr_{\gamma,\bf{O}^{\mathrm{reg}}}$ being the open dense stratum.

It is therefore natural to ask what the closed strata are. Conjecture~\ref{conj of Yun} predicts that there is a unique closed stratum corresponding to the minimal reduction.

Theorem~\ref{thm.main} has further applications. 
By \cite[Lemma~11.6]{Yun21}, if Conjecture~\ref{conj of Yun} holds, then $L^\heartsuit \mathfrak{c}$ admits a stratification whose strata can be regarded as affine analogs of special pieces in the nilpotent cone $\cN$.

In \cite{JY23}, Jakob and Yun studied the isoclinic Deligne--Simpson problem and reduced it to Conjecture~\ref{conj of Yun} for homogeneous elements in $L^\heartsuit \g$. 
They verified the conjecture in the classical cases for homogeneous elements and explicitly determined the minimal reductions. Our construction in Theorem~\ref{thm.main} is compatible with and recovers their results.

\subsection{Plan of the paper}
This paper is organized as follows.

In Section~\ref{Preli}, we recall basic facts on reduction maps, nilpotent orbits, and Newton polygons. Subsection~\ref{linear algebra} is devoted to linear algebraic preliminaries, concerning vector spaces equipped with non-degenerate pairings and (anti-)self-adjoint operators. The results established there will be used repeatedly in the analysis of types~B and~D.

In Section~\ref{type AC}, we treat the cases of types~A and~C. We first construct the minimal reduction following the approach of \cite{Yun21}. As a warm-up for the subsequent section, we then present an alternative construction based on Newton polygons, whose key step is to conjugate a rational normal form to one with minimal reduction. In the case of type~C, this construction further requires the explicit construction of a non-degenerate pairing.

Section~\ref{type BD} is devoted to the most delicate cases, namely types~B and~D. In Subsection~\ref{find the reduction}, we introduce the notion of admissible partitions and prove that, among all admissible partitions dominating the type~A minimal reduction, there is a unique candidate which, if it occurs as a reduction, must be minimal. In Subsections~\ref{single slope} and~\ref{subsec:modify pairing}, we show that this partition indeed arises as the reduction of $\gamma$. Combined with Corollary~\ref{very even}, this completes the proof of the uniqueness of minimal reductions in types~B and~D.

\vspace{0.3cm}
\paragraph{\textbf{Notations}}
Throughout this paper, we adopt the following notations.
\begin{itemize}
    \item The symbol $\equiv$ always means congruence modulo $2$, i.e.\ $\equiv \pmod{2}$.
    \item We denote by $\mathcal{O}=\mathbb{C}[\![t]\!]$ the ring of formal power series in one variable $t$, and by $\sK$ its field of fractions.
    \item For an element $a\in \mathcal{O}$, we write $(a)_i$ for the coefficient of $t^i$ in $a$.
\end{itemize}

\section*{Acknowledgments}

We would like to thank Professor Yun for many helpful discussions. 
Part of this work was completed during B.~Wang’s postdoctoral appointment at the Chinese University of Hong Kong, and he would like to thank Michael McBreen for his support. 
Part of this work was completed during a visit of X.~Wen to KIAS, and he gratefully acknowledges the warm hospitality of KIAS. 
Y.~Wen would like to thank Woonam Lim for his support during the author’s stay at Yonsei University, where this work was carried out.

\section{Preliminary}\label{Preli}

\subsection{Minimal Reduction Map}
We denote the centralizer of $\gamma$ in $LG$ by $LG_{\gamma}$ which can also be treated as a torus over $\Spec\sK$.
\begin{lemma}
    The loop centralizer $LG_{\gamma}$ acts on $\Gr_{\gamma,\bf{O}}$ for any $\bf{O}\in\RT(\gamma)$.
\end{lemma}
\begin{proof}
    There is a natural action of $LG_{\gamma}$ on $\Gr_{\gamma}$:
    \[
    gL^+G\mapsto hgL^+G,\; h\in LG_{\gamma}
    \]
    by left multiplication. Notice that we have the following commutative diagram:
    \[
    \begin{tikzcd}
         gL^+G\ar[d,"h\cdot"]\ar[r,"ev_{\gamma}"]& \frac{\sN}{G}\ar[d,"Id"]\\
         hgL^+G \ar[r,"ev_{\gamma}"]&\frac{\sN}{G}
    \end{tikzcd}
    \]
    since $h\in LG_{\gamma}.$
\end{proof}
\begin{lemma}\label{lem:congjuation invriance of RT}
    If $\gamma,\gamma'$ are in the same conjugacy class of $LG$, then $\RT_{\min}(\gamma)=\RT_{\min}(\gamma')$.
\end{lemma}
\begin{proof}
    Choose $h\in LG$ such that $h^{-1}\gamma'h=\gamma$, then $g^{-1}\gamma g=(gh)^{-1}\gamma' gh\in L^+\g$. We have an isomorphism:
    \[
    \Gr_{\gamma}\rightarrow \Gr_{\gamma'}, \;\;\;\;
        gL^+G\mapsto hgL^+G 
    \]       
    which induces an isomorphism between $\Gr_{\gamma,\bf{O}}\xrightarrow[h\cdot]{\cong}\Gr_{\gamma',\bf{O}}$ for each nilpotent orbit $\bf{O}$.
\end{proof}
\begin{lemma}\label{lem:single conj class}
    For $\gamma,\gamma'\in L^{\heartsuit}\g$, if $\chi(\gamma)=\chi(\gamma')$, then they are conjugate by $LG$, i.e., over $\sK$.
\end{lemma}
\begin{proof}
    Since $\chi(\gamma)=\chi(\gamma')$ and they are regular semisimple, there exists $g\in G(\bar{\sK})$ such that $\gamma'=g\gamma g^{-1}$. For any $\sigma\in \Gal(\bar{\sK}/\sK)$, we have:
    \[
    \sigma(g)\gamma \sigma(g)^{-1}=g\gamma g^{-1}
    \]
    Hence $g^{-1}\sigma(g)\in LG_{\gamma}(\bar{\sK})$. Here we treat $LG_{\gamma}$ as a torus over $\sK$.
    
    We can define a 1-cocyle:
    \[
    \sigma\mapsto g^{-1}\sigma(g)\in LG_{\gamma}(\bar{\sK}).
    \]
    Since $\sK$ is a complete Henselian field with residue field algebraically closed and $LG_{\gamma}$ is a torus over $\sK$, $H^1(\Gal(\bar{\sK}/\sK), LG_{\gamma})=0$. Hence, $\exists h\in LG_{\gamma}(\bar{\sK})$ such that $g^{-1}\sigma (g)=h^{-1}\sigma(h)$ for any $\sigma\in \Gal(\bar{\sK}/\sK)$. Notice that $\gamma'=(gh^{-1})\gamma(hg^{-1})$ and $gh^{-1}\in LG_{\gamma}(\sK)$. The lemma follows.
\end{proof}
Combining with Lemma~\ref{lem:congjuation invriance of RT}, we have the following corollary.
\begin{corollary}\label{cor:only depends on char}
   The reduction type $\RT(\gamma)$ only depends on $\chi(\gamma)$. 
\end{corollary}

\subsection{Nilpotent orbits}

Nilpotent orbits of classical Lie algebras are parametrized by partitions with suitable parity conditions; see \cite{CM93} for a detailed account. 
Throughout this paper, we freely identify a nilpotent orbit with the partition labeling it.

Let $\cP(N)$ denote the set of partitions $\bf{d}=[d_1\geq d_2\geq \cdots]$ of $N$. 
For $\varepsilon=\pm1$, define
\[
\cP_\varepsilon(N)
=
\Bigl\{\bf{d}\in \cP(N)\ \Big|\ 
\#\{j\mid d_j=i\}\equiv 0 \ \text{for all } i \text{ with } (-1)^i=\varepsilon
\Bigr\}.
\]
It is well known that nilpotent orbits in $\mathfrak{sl}_N$, $\mathfrak{so}_{2n+1}$, and $\mathfrak{sp}_{2n}$ are in bijection with $\cP(N)$, $\cP_1(2n+1)$, and $\cP_{-1}(2n)$, respectively. 

For $\mathfrak{so}_{2n}$, the parametrization is given by $\cP_1(2n)$, except that partitions with all parts even correspond to two distinct $\SO_{2n}$-orbits (the so-called \emph{very even} orbits), denoted by $\bf{O}^\mathrm{I} \sqcup \bf{O}^\mathrm{II}$. 

The dominance order on partitions,
\[
\bf{d}\ge \bf{f}
\quad\Longleftrightarrow\quad
\sum_{j=1}^k d_j \ge \sum_{j=1}^k f_j \quad \text{for all } k,
\]
induces a partial order on nilpotent orbits, which agrees with the closure order: 
$\overline{\bf{O}}_{\bf{d}}=\bigsqcup_{\bf{f}\le \bf{d}}\bf{O}_{\bf{f}}$.

\subsection{Newton Polygons}

For $G=\GL_n,\SL_n,\Sp_{2n},\SO_{2n}$, and $\gamma\in L^{\heartsuit}\g$, we put:
\[
    \chi_{\gamma}(\lambda)=\det(\lambda-\gamma)\in\sO[\lambda]
\]
But in the case $G=\SO_{2n}$, we keep track with the pfaffian of $\gamma$ rather than its determinant.

For $G=\SO_{2n+1}$, we put:
\[
    \chi_{\gamma}(\lambda)=\det(\lambda-\gamma)/\lambda\in \sO[\lambda].
\]
We denote by $\NP(\gamma)$ the Newton polygon of $\chi(\lambda)$. For example, when  
\[
 \chi(\lambda)=\lambda^{n}+a_1\lambda^{n-1}+\cdots +a_{n-1}\lambda +a_n
\]
then $\NP(\gamma)$ is defined as the convex hull of $\{(\ord(a_{i}),n-i)\}_{i=0}^{n}$ where $a_0=1$ and $\ord (0)=+\infty$. 

\begin{lemma}\label{decomposition wrt NP}
    If the Newton polygon $\NP(\gamma)$ consists of edges with slopes $\mu_1> \cdots >\mu_r$, then we have a decomposition: $$\chi_{\gamma}(\lambda)=\prod_{i=1}^rf_i(\lambda)$$ where the Newton polygon of $f_i(\lambda)$ consists only one edge with slope $\mu_i$.
\end{lemma}

\subsection{Pairings and Lie algebras in type B, C and D}\label{subsec:preparation modify pairing}
In this subsection, we aim to clarify our explicit construction of minimal reduction. In this paper, we regard $L^+\mathfrak{sp}_{N}$ and $L^+\mathfrak{so}_{N}$ as subalgebras of $L^+\mathfrak{gl}_N$, by fixing a non-degenerate (skew-)symmetric pairing on $\mathcal{O}^N$. Thus elements in $L^+\mathfrak{sp}_N$ or $L^+\mathfrak{so}_N$ can be regarded as elements in $\mathfrak{gl}_N$ which are anti-self-adjoint with respect to the pairing. 

However, for the convenience of constructing the minimal reduction, we need to conjugate the element $\gamma$ with elements not necessarily in $L^+\operatorname{Sp}_N$ or $L^+\operatorname{SO}_N$. In other words, we change the pairing to find a relatively simple matrix expression. We now explain why this still works. First notice that:
\begin{lemma}
    All the nondegenerate symmetric (resp. skew-symmetric) pairing on $\sO^N$ are equivalent, i.e., up to a choice of a basis.
\end{lemma}
More precisely, denote the originally fixed pairing by $g$ (which will be explicitly given later). Suppose that we find an element $h\in L\GL_N$ such that $h\gamma h^{-1}$ achieves ``the minimal reduction" and is compatible with a new pairing $g'$ on $\sO^{N}$(which will be explicit later). By the above lemma, we can find $y\in L^+\GL_N$ such that $y^tg'y=g$. Then $yh\gamma (yh)^{-1}$ is compatible with $g$. 

For $\Sp_N$, since $yh\gamma (yh)^{-1}$ shares the same characteristic polynomial with $\gamma$, then by Corollary \ref{cor:only depends on char}, it is conjugate to $\gamma$ in $L\Sp_N$, we obtain the minimal reduction. 

For $\SO_N$,  $yh\gamma (yh)^{-1}$ is conjugate to $\gamma$ in $L\mathrm{O}_N$. When the minimal reduction is not very even, the conjugation by $L\mathrm{O}_N$ suffices to determine the minimal reduction. For very even cases, we postpone to Corollary \ref{very even}.

\subsection{Companion Matrix and Pairings}\label{linear algebra}

In this part, we consider a vector space $V$ over $\mathbb{C}$. Assume that we have a non-degenerate symmetric pairing $g$ on $V$ and a \emph{regular} invertible linear transform $\Phi: V\rightarrow V$. If $g(\Phi u,v)=g(u,\Phi v)$, we say that $\Phi$ is self-adjoint (with respect to $g$) and if $g(\Phi u,v)=-g(u,\Phi v)$, we say that $\Phi$ is anti-self-adjoint. We use $\alpha(X)$ to denote the characteristic polynomial of $\Phi$.

\begin{lemma}\label{square of polynomial}
    Suppose that $\Phi$ is self-adjoint and $\dim V=2n$. There exist a maximal isotropic subspace $W$ of $V$ such that $\Phi (W)=W$ if and only if $\alpha(X)=\beta(X)^2$ for some polynomial $\beta(X)$. If so, it is unique. 
\end{lemma}
\begin{proof}
    Suppose there exists such a maximal isotropic subspace $W$. The nondegenerate pairing $g$ induces an isomorphism:
    \[
    (V/W, \Phi)\rightarrow (W^*, \Phi^{\vee})
    \]
    Hence the characteristic polynomial of $\Phi$ on $V$ should be $(\Phi|_{W})^2$.

    Conversely, suppose that $\alpha(X)=\beta(X)^2$. Since $\Phi$ is self-adjoint, the image $\beta(\Phi)V\subset V$ is isotropic. 
    Notice that $\Phi$ is a regular transform, hence its minimal polynomial coincides with $\alpha(X)$. Since $\alpha(X)=\beta(X)^2$. Considering its Jordan normal form, we have $\rk \beta(\Phi)=1/2\dim V$.

    We now show that $\beta(\Phi)V$ is the unique choice. In fact, for any such $W$, since it is stable under $\Phi$, the characteristic polynomial of $\Phi|_{W}$ can only be $\beta(x)$. In particular $W\subset \ker \beta(\Phi)= \beta(\Phi)V$.
\end{proof}
The following lemma can be proved similarly.
\begin{lemma}\label{self dual polynomial}
    Suppose that $\Phi$ is anti-self-adjoint and $\dim V=2n$. There exist a maximal isotropic subspace $W$ of $V$ such that $\Phi (W)=W$ if and only if $\alpha(X)=\beta(X)\beta(-X)$ for some polynomial $\beta(X)$.
\end{lemma}

We denote the Jordan normal form of $\Phi$ by $\mathrm{diag}(J_i)$ where the eigenvlaues of $J_i$ are denoted as $\lambda_i$ and $J_i$ is of size $m_i$. In particular, $\lambda_i$ are pairwise different and we decompose $V=\oplus_{i}V_i$ where each $V_i$ is the generalized $\lambda_i$-eigenspace. 
\begin{lemma}\label{lem:canonical metric}
    Given a vector space $U$ of dimension $m$ and $J\in \End(U)$ with single Jordan block. Let $g$ be a nondegenerate symmetric pairing on $U$ such that $g(Ju,v)=g(u,Jv)$. Then there exists a basis $\{e_i\}$ with $J(e_i)=\lambda_i e_i+e_{i-1}$ and $g(e_i, e_{m+1-i})=1$ and $g(e_i,e_j)=0, j\ne m+1-i$.
\end{lemma}
\begin{proof}
    Clearly, we can find a basis satisfying the first condition, denoted by $\{f_i\}$. Since $g(Ju,v)=g(u,Jv)$, we can check that 
    \begin{enumerate}
        \item $g(f_i,f_j)=g(f_l,f_k)$ if $i+j=l+k> m$,
        \item $g(f_i,f_j)=0$ if $i+j\le m$.
    \end{enumerate}
    Without loss of generality, we may assume that $g(f_i,f_{m+1-i})=1$. Now we only need to solve the following equations for $\{a_i\}$:
    \[
    \begin{pmatrix}
        0&0&0\ldots&0&1\\
        0&0&0\ldots&1&*\\
        &&\ldots\\
        0&1&*\ldots&*&*\\
        1&*&*\ldots&*&*
    \end{pmatrix}=A^t\begin{pmatrix}
        0&0&0\ldots&0&1\\
        0&0&0\ldots&1&0\\
        &&\ldots\\
        0&1&0\ldots&0&0\\
        1&0&0\ldots&0&0
    \end{pmatrix}A
    \]
    where 
    \[
    A=\begin{pmatrix}
        1&a_1&a_2&a_3&\ldots&a_{m-1}&a_m\\
        0&1&a_1&a_2&\ldots&a_{m-2}&a_{m-1}\\
        &&&&\ldots&\\
        0&0&0&0&\ldots& 1&a_1\\
        0&0&0&0&\ldots& 0&1
    \end{pmatrix}
    \]
    It is straightforward to check that it has (finitely many) solutions.
\end{proof}
We may denote the pairing by $g_0$  where $g_0(e_i, e_{m-i})=1$ and $g_0(e_i,e_j)=0, j\ne m-i$. The following corollary is immediate.
\begin{corollary}\label{cor:canonical split}
    Suppose that $\Phi$ is self-adjoint, then there exists a choice of a basis such that:
    \[
    (V,\Phi,g)=\oplus(V_i,J_i,g_{0,i})
    \]
\end{corollary}
Now we are ready to prove the following.
\begin{proposition}\label{prop:codim 1}
     Suppose that $\Phi$ is self-adjoint and $\dim V=2n$. There exists a maximal isotropic subspace $W$ of $V$ such that $\codim_{W}\Phi(W)\cap W=1$.
\end{proposition}

\begin{proof}
   By Corollary \ref{cor:canonical split} and Lemma \ref{square of polynomial}, it suffices to construct such maximal isotropic subspaces for:
   \[
   \oplus_{\dim V_i\equiv 1} (V_i,J_i,g_{0,i}).
   \]
   For each $V_i$, there is clearly an isotropic subspace $W_i$ of dimension $(\dim V_i-1)/2$ which is stable under $J_i$. Now consider the following construction:
   \[
   (\oplus W_i^{\perp}/W_i, J_i=\lambda_i, g_{0,i}=1)
   \]
   It suffices to construct a maximal isotropic subspace $V$ of $\oplus W_i^{\perp}/W_i$ such that $\codim_V (\mathrm{diag}(\lambda_i)(V)\cap V)=1$.

   Since $\lambda_i$ are pairwise distinct, then by Lemma \ref{square of polynomial}, for there is no nonzero isotropic subsapce of $\oplus W_i^{\perp}/W_i$ stable under $\psi:=\mathrm{diag}(\lambda_i)$. Let us put $\dim\oplus W_i^{\perp}/W_i=2\delta$. To construct such a $V$, we only need to find $v\in \oplus W_i^{\perp}/W_i$ such that $\left\langle v,  \psi v, \ldots, \psi^{\delta-1}v\right\rangle$ is isotropic which does exist. 
\end{proof}
 
\begin{lemma}\label{even maximal isotropic}
    Suppose that $\Phi$ is self-adjoint and $\dim V=2n$. Then we have an isotropic subspace $V_{0}$ of dimension $n-1$, and two vectors $v^{-}$ and $v^{+}$, satisfying the following conditions:
    \begin{enumerate}
        \item The pairing of $v^{\pm}$ with $V_{0}$ is zero, and $g(v^{\pm}, v^{\pm})=1$, $g(v^{\pm}, v^{\mp})=0$;
        \item  $\left\langle v^{-}\right\rangle+V_0=\Phi\left(V_0+\left\langle v^{+}\right\rangle\right)$.
    \end{enumerate}
\end{lemma}
\begin{proof}
   Assuming that all $V_i$ are of even dimension, then  we can find an isotropic subspace of dimension $n-1$ which is stable under $\Phi$. We denote it by $V_0$ and by $V_0^{\perp}$ the orthogonal complement of $V_0$. In particular, $v^{\pm}\in V_0^{\perp}$ and $\dim V_0^{\perp}/V=2$. Since $\Phi$ is invertible, to satisfy Condition 2, it suffices to show that $v^{-}\in \Phi(V_0+\langle v^+\rangle)$. 
   Since $V_0$ is isotropic and invariant under $\Phi$, it suffices to construct $v^{\pm}$ for  $(V_0^{\perp}/V_0,\Phi)$, i.e., 2-dimensional cases which can be done easily.

  Suppose not. For each $V_i$ in the decomposition in Lemma \ref{lem:canonical metric}, we can find an isotropic subspace $V_{i,0}$ of dimension $\lfloor\frac{\dim V_i}{2}\rfloor$ which is stable under $\Phi$. Then it suffices to construct for  $(\oplus \frac{V^{\perp}_{0,i}}{V_{0,i}}, \oplus (\lambda_i))$ where $\lambda_i$ are eigenvalues of $\Phi$. Hence we may assume that $(V,\Phi)=\oplus_{i=2n} (V_i,\lambda_i)$ where each $\dim V_i=1$. It suffices to find a vector $v_{+}$, such that the triple we want is $(v^+, V_0=\langle \Phi v^{+},\ldots, \Phi^{n-1}v^+\rangle,\Phi^{n}v^+)$. Notice that $g(v^+,V_0)=g(v^{-},V_0)=g(V_0,V_0)=0$ has detemined the $\langle v_0\rangle$ uniquely, and for such $v_0$, it satisfies $g(v^+,v^+)\ne 0, g(\Phi^n v^+,\Phi^n v^+)\ne 0$ automatically.
\end{proof}
The proof and construction of the following lemma is similar to the second part of the previous lemma.
\begin{lemma}\label{odd maximal isotropic}
    Suppose that $\Phi$ is self-adjoint and $\dim V=2n+1$. Then we have an isotropic subspace $V_{0}$ of dimension $n$, and two vectors $v^{\pm}$, satisfying the following conditions: 
    \begin{enumerate}
        \item The pairing of $v^\pm$ with $V_{0}$ is zero, and $g(v^\pm, v^\pm)=1$;
        \item $V_{0}\subset \Phi (V_{0}+\left\langle  v^+\right\rangle)$;
        \item $\Phi (V_{0})\subset \left\langle  v^-\right\rangle + V_{0}$.
    \end{enumerate}
\end{lemma}

\section{Construction of Minimal Reduction: Type A and C}\label{type AC}

\subsection{Construction via Skeleton}

We first provide a construction of minimal reduction in types A and C following the idea of Yun\cite{Yun21}. A closely related argument already appeared in~\cite{JY23}, although restricted there to elements whose Newton polygon has a single slope. In the subsequent subsection, we give an alternative construction whose strategy will also apply to special orthogonal groups.

\begin{definition}\label{def:skeleton}
    For each conjugacy class $[w]$ of the Weyl group, let $LT_{[w]}$ be a maximal torus of $G$ over $\sK$ corresponding to $[w]$. 
    Let $L^+T_{[w]}$ be the unique solvable parahoric subgroup of $LT_{[w]}$. 
    The \emph{skeleton} is defined to be
    \[
        \sX^{[w]} := (\Gr_G)^{L^+T_{[w]}}_{\mathrm{red}}.
    \]
\end{definition}

The relation between skeletons and minimal reductions is given by the following result.

\begin{lemma}[\cite{Yun21}, Lemma~7.3]
    For any $\gamma\in L\g_{[w]}\cap L^{\heartsuit}\g$, we have
    \[
        \sX^{[w]}\cap \Gr_{\gamma,\bf{O}}\neq\emptyset,
    \]
    where $\bf{O}\in \RT_{\min}(\gamma)$.
\end{lemma}

For a Levi sugroup $M$, we have the following natural map:
\[
\iota^{M}:\{\text{nilpotent orbits in}\; M\}\rightarrow \{\text{nilpotent orbits in}\; G\}.
\]
 Now we recall the following useful lemma by Yun.
\begin{lemma}\cite[Lemma 3.2]{Yun21}\label{lem:Levi}
    If $\gamma\in L\mathfrak{m}$ with $M\subset G$ a Levi subgroup, then $\iota^{M}(\RT_{\min}^{M}(\gamma))=\RT_{\min}(\gamma)$. 
\end{lemma}
As a result, we have
\begin{corollary}\label{cor:elliptic cases}
    It suffices to consider the minimal reduction of such $\gamma$:
\begin{enumerate}
    \item For Type A, $\chi_{\gamma}(\lambda)$ is irreducible;
    \item For Type B, C, D, each irreducible factor $f_i(\lambda)$ of $\chi_{\gamma}$ is self-dual, i.e., $f_i(\lambda)=f_i(-\lambda)$.
\end{enumerate}
\end{corollary}

\subsubsection{Type A}

For a given $\gamma\in L\gt_{[w]}\cap L^{\heartsuit}\g$, suppose that its characteristic polynomial can be decomposed into irreducible factors:
\[
f=\prod_{i} f_i
\]
Each irreducible $f_i\in\sK[\lambda]$ defines an extension $E_i$ of $\sK$ with with $[E_i:\sK]=n_i$. 

\begin{lemma}\label{lem:m power}
    Let $J$ be the regular Jordan normal form of size $n$. Then the partition corresponding to $J^m$ is the $m$-balanced partition of $n$.
\end{lemma}
\begin{proof}
    Let $V$ be $n$-dimensional vector space, with basis $\{e_i\}$ such that:
    \[
    Je_{i}=e_{i-1}.
    \]
    Then $J^me_{i}=e_{i-m}$. The lemma follows.
\end{proof}

\begin{proposition}\label{prop:A}
    Suppose that each constant term of $f_i$ has evaluation (with respect to $\sK$) $m_i$. Then the partition of $\RT^{\mathrm{A}}_{\min}(\gamma)$ is the collection of $m_i$-balanced partition of $n_i$.
\end{proposition}
\begin{proof}
     Firstly, recall that the skeleton $\sX^{[w]}$ is the set of $\oplus\sO_{E_i}$-submodules of $\oplus E_i$. It is simple to check that:
    \[
    \ord_{\sO_{E_i}}\lambda=m_i.
    \]
     Hence $\lambda|_{\sO_{E_i}}= \pi^m_{i}\phi_i(\pi_i)$ where $\phi(\pi_i)$ is a unit in $\sO_{E_i}$. Notice that $\pi_i$ has partition $[n_i]$ acting on $\Lambda/t\Lambda$. Then the result follows from Lemma \ref{lem:m power} and Lemma \ref{lem:Levi}.
\end{proof}
\subsubsection{Type C}For a given $\gamma\in LT_{[w]}\cap L^{\heartsuit}\g$, suppose that its characteristic polynomial can be decomposed into irreducible factors:
\[
f=\prod_{i} f_i
\]

\begin{enumerate}
    \item $f_i(\lambda)=f_i(-\lambda)$, i.e., degree $f_i$ is even and we say it is self-dual. 
    \item For pair of $f_i, f_{i'}$ with $f_{i}(\lambda)=(-1)^{\deg f_i}f_{i'}(-\lambda)$.
\end{enumerate}
For each pair of irreducible factors (i.e., the second case), we can find a Levi subalgebra $\mathfrak{m}$ such that $\gamma\in L\mathfrak{m}\cap L^{\heartsuit}\g$. Notice that the Levi subalgebra of $\g$ are products of type A with a single type C. By Lemma \ref{lem:Levi}, it suffices to consider that all irreducible factors of $f$ are self-dual.

Hence, from $\gamma$, we can construct a degree 2n extension of $\sK$:
\[
E=\oplus E_i
\]
with $[E_i:F_i]=2$, $[F_i:\sK]=n_i$. $E$ is endowed with a natural symplectic pairing, for each $i$, we denote by $\pi_i$ uniformizer and $\mathfrak{d}_i$ generator of the different ideal. Then we define a symplectic pairing:
\[
    \omega_i(x,y)=\Tr_{F_i/\sK}\left(\mathfrak{d}_i\frac{x\sigma(y)-y\sigma(x)}{\pi_i}\right)
\]
which restricts to a non-degenerate symplectic pairing on $\oplus \sO_{E_i}$. The $\mathbb{C}$-points of $\Gr_{G}$ are self-dual $\sO$-lattices in $E$.

\begin{proposition}\label{prop:C}
    Suppose that each constant term of $f_i$ has evaluation (with respect to $\sK$) $m_i$. Then the partition of $\RT^{\mathrm{C}}_{\min}(\gamma)$ is the collection of $m_i$-balanced partition of $2n_i$.
\end{proposition}
\begin{proof}
    Firstly, recall that the skeleton $\sX^{[w]}$ is a singleton $\oplus\sO_{E_i}$, see Yun\cite{Yun21}. Since the nondegenerate pairing on the lattice is given by $\oplus\omega_i$, it suffices to work out the minimal reduction for each $\sO_{E_i}$ factors. It is simple to check that:
    \[
    \ord_{\sO_{E_i}}\lambda=m_i.
    \]
    Hence $\lambda|_{\sO_{E_i}}= \pi^m_{i}\phi_i(\pi_i)$ where $\phi(\pi_i)$ is a unit in $\sO_{E_i}$. Notice that $\pi_i$ has partition $[2n_i]$ acting on $\sO_{E_i}/t\sO_{E_i}$. Then the result follows from Lemma \ref{lem:m power} and Lemma \ref{lem:Levi}.
\end{proof}

\subsection{Construction via Newton Polygon}

Let $\g=\mathfrak{gl}_n$ or $\mathfrak{sp}_{2n}$. Given $\gamma_0\in L^{\heartsuit}\g$, we write $f(\lambda)=\chi_{\gamma_0}(\lambda)$. We decompose it into  factors whose Newton polygon has a single slope:\footnote{For our explicit conjugation, we do not need each part to be irreducible.}
\[
f=\prod _{i=1}^{r}f_i
\]
i.e., the associated Newton polygon $\NP_i$ has a single slope, and the slopes are different for each $\NP_i$.
We can check that $\operatorname{R}(f)$ is conjugate to $\mathrm{diag}\{(\operatorname{R}(f_1),\ldots, \operatorname{R}(f_{r})\}$ which lies in the loop algebra of a Levi subalgebra.

\subsubsection{Type A}

By Corollary \ref{cor:only depends on char}, Lemma \ref{decomposition wrt NP} and Lemma \ref{lem:Levi} , in type~A we may reduce to the case where $f$ is irreducible; in particular, $f$ has a single slope.

Now, assume that $f=\lambda^n+a_1\lambda^{n-1}+\ldots +a_{n}$ with $\ord_ta_n=m$. Since it has a single slope, we have
\[
\ord_t a_i \ge \left\lceil \tfrac{i \cdot m}{n} \right\rceil.
\]

To prove Proposition \ref{prop:A}, it suffices to establish the following statement.

\begin{proposition}\label{prop:minimal is balanced}
    Assume $\chi_{\gamma}(\lambda)=f$ as above. Then $\RT^{\mathrm{A}}_{\min}(\gamma)$ is the associated balanced partition $[(k+1)^{l}, k^{m-l}]$.
\end{proposition}
\begin{proof}
    We may take $\gamma=\operatorname{R}(f)$. We define
    \[
        P=\mathrm{diag}\left(\{t^{m-1}\}^{k}, \cdots , \{t^{l}\}^{k}, \{t^{l-1}\}^{k+1}, \cdots , \{t\}^{k+1}, \{1\}^{k+1}\right).
    \]
   here $\{t^{i}\}^{k}$ means $t^i$ appears $k$ many times. By $\ord a_i \ge \left\lceil \tfrac{i \cdot m}{n} \right\rceil$, we can check that:
    \[
        P^{-1}\gamma P
    \]
    lies in $L^+\g$ and has reduction $[(k+1)^{l}, k^{m-l}]$.

    Conversely, for any partition smaller than the balanced one, the generic Newton polygon is higher than $\NP(f)$ which is impossible.
\end{proof}

\subsubsection{Type C}

Assume $\g=\mathfrak{sp}_{2n}$. Then $f(\lambda)=f(-\lambda)$ and $f_{i}(\lambda)=f_{i}(-\lambda)$. Write
\[
    \deg f_i=2n_i,\qquad \ord_t(f_i(0))=m_i,\qquad 2n_i=m_i k_i+l_i,\ \ 0\le l_i<m_i.
\]
Let $\gamma_i = R(f_i)$. We treat each factor separately, and for notational simplicity, suppress the index $i$ in the construction below.

Let $P=\mathrm{diag}(t^{p_1},\ldots,t^{p_{2n}})$, where
\[
    p_{j}=
    \begin{cases}
    m-\left\lceil\frac{j m}{2n}\right\rceil, & 1\le j\le n,\\[4pt]
    \left\lceil\frac{(2n+1-j)m}{2n}\right\rceil-1, & n+1\le j\le 2n.
\end{cases}
\]
In particular,
\begin{equation}\label{eq:symmetry}
    p_{j}+p_{2n+1-j}=m-1.
\end{equation}
\begin{lemma}\label{lem:integrality of modified gamma C}
    $P^{-1}\gamma P$ is integral.
\end{lemma}
\begin{proof}
    First, we can verify that, for $1\le j\le 2n$, $p_{j}\le \ord_t a_{2n+1-j}-1$, i.e.,
    \[
    m-\left\lceil\frac{jm}{2n}\right\rceil \le \left\lceil\frac{(2n+1-j)m}{2n}\right\rceil-1.
    \]
    It holds since $\left\lceil\frac{jm}{2n}\right\rceil-\left\lfloor\frac{(j-1)m}{2n}\right\rfloor\ge 1$. Hence $P^{-1}\gamma P$ is integral and all elements of $\gamma(0)$ are 0 except possibly those $\gamma(0)_{i,i-1}, 2\le i\le n$.
\end{proof}
\begin{lemma}\label{lem:partition of modified gamma C}
     The reduction $(P^{-1}\gamma P)(0)$ has partition $[(k+1)^{l},k^{m-l}]$.
\end{lemma}
\begin{proof}
    By the previous lemma, $P^{-1}\gamma P$ is integral and all elements of $(P^{-1}\gamma P)(0)$ are 0 except possibly those $\gamma(0)_{i,i-1}, 2\le i\le n$.
    
    We first show that 
    \begin{equation}\label{eq:sequence length}
    \#\{j \mid \ord P_j=l\}=k\; \text{ or }\; k+1 \quad \text{for}\quad 0\le l\le m-1.
    \end{equation}
    Since $m$ is even, it suffices to check this for the sequence $\left\{\left\lceil\tfrac{jm}{2n}\right\rceil\right\}_{1\le j\le n}$.
    Now suppose that $\left\lceil\tfrac{(j-1)m}{2n}\right\rceil<\Delta=\left\lceil\tfrac{jm}{2n}\right\rceil=\ldots=\left\lceil\tfrac{(j+\delta)m}{2n}\right\rceil<\ldots$. 
    
    Clearly, $\delta<k+1$. Since $\left\lceil\tfrac{(j-1)m}{2n}\right\rceil<\Delta=\left\lceil\tfrac{jm}{2n}\right\rceil$, $2n\Delta-jm\ge 2n-m$. 
    If $\delta\le k-2$, $2n\Delta-(j+\delta+1)m\ge 2n-(\delta+1)m>0$, a contradiction. Hence we prove \eqref{eq:sequence length}.
    
    As a result, only $k$ or $k+1$ appears in the partition of $(P^{-1}\gamma P)(0)$. Since the orders (of $P_j$'s) range from $0$ to $m-1$, the partition has to be $[(k+1)^{l}, (k)^{m-l}]$.
\end{proof}
We now determine where the orders ``jump".
\begin{lemma}\label{lem:jumps}
    We have $p_i>p_{i+1}$ precisely when
    \[
        i=j k+\left\lfloor\frac{j l}{m}\right\rfloor,\qquad 1\le j\le \frac{m}{2},
    \]
    or
    \[
        i=j k+\left\lceil\frac{j l}{m}\right\rceil,\qquad \frac{m}{2}+1\le j\le m-1.
    \]
\end{lemma}
\begin{proof}
    We first determine when $\left\lceil\tfrac{im}{2n}\right\rceil<\left\lceil\tfrac{(i+1)m}{2n}\right\rceil$. We rewrite $i=jk+\delta, 0\le\delta<k$. Now we have $\Delta=\left\lceil\tfrac{(jk+\delta)m}{2n}\right\rceil<\left\lceil\tfrac{(jk+\delta+1)m}{2n}\right\rceil=\Delta+1$. In other words,
    \[
        \frac{(jk+\delta)m}{2n}\le \Delta<\frac{(jk+\delta+1)m}{2n}.
    \]
    Hence, we have
    \[
        \delta=(\Delta-j)k+\left\lfloor\tfrac{\Delta l}{m}\right\rfloor.
    \]

    Notice that $p_{n}>p_{n+1}$ and for $j=\frac{m}{2}$, we have $jk+\left\lfloor\tfrac{jl}{m}\right\rfloor=n$. Hence, we verify the lemma for $1\le i\le n$. 

    Similarly, for $n+1\le i\le 2n$, we have $i=2n-jk-\left\lfloor\tfrac{jl}{m}\right\rfloor$ for $1\le j\le \frac{m}{2}-1$. And the Lemma follows.
\end{proof}

Now we need to construct skew-symmetric matrices $g$ so that $g(P^{-1}\gamma P)=-(P^{-1}\gamma P)^\vee g$. 

First, we define a pairing $\widetilde g$ on the module $\sE=\frac{\sO[\lambda]}{(f(\lambda))}$, with a basis $\{1, \lambda, \cdots , \lambda^{2n-1}\}$. Hence, the action of $\lambda$ on $\sE$ has matrix $\operatorname{R}(f)$. And we define 
\[
    \tilde{g}(\lambda^\alpha, \lambda^\beta)=\operatorname{Tr}\left(\frac{\lambda^{\alpha}\sigma(\lambda^{\beta})}{f^\prime(\lambda)}\right)
\]
as in Section 3.2 of \cite{TN24}, here $\sigma(\lambda^\beta)=(-1)^{\beta} \lambda^\beta$ is the involution. One has $\tilde{g}\operatorname{R}(f)=-\operatorname{R}(f)^\vee \tilde{g}$. 

Moreover, we have the following estimate.

\begin{lemma}\label{pairing g in skew case}
    $\tilde{g}$ is non-degenerate skew-symmetric, and $\ord_t(\tilde{g})_{i, j}\geq \ord_t a_{i+j-2n-1}$ for $i+j\geq 2n+1$.
\end{lemma}

\begin{proof}
    Notice that $(\tilde{g})_{i, j}=0$ for $i+j\leq 2n$ and $(\tilde{g})_{i, j}=\pm 1$ for $i+j=2n+1$, hence $\tilde{g}$ is non-degenerate. To see that $\tilde{g}$ is  skew-symmetric, it is enough to notice that $(\tilde{g})_{i, j}=0$ if $i+j\equiv 0$ since $f(\lambda)$ is self-dual.

    Now we compute $\ord_t(\tilde{g})_{i, j}$. Notice that $(\tilde{g})_{i, j}=(-1)^{j+\beta}(\tilde{g})_{\alpha, \beta}$ whenever $i+j=\alpha+\beta$. So we compute $(\tilde{g})_{2i+1, 2n}$ for $1\leq i \leq n-1$.

    By direct computation, we have $(g_0)_{3, 2n}=a_2$ and the following recurrence relation: $$(g_0)_{2i+1, 2n}=a_{2i}-\sum_{j=1}^{i-1}a_{2j}\cdot (g_0)_{2i-2j+1, 2n}.$$

    To estimate the order of $(g_0)_{2i+1, 2n}$, we induct on $i$ and use the recurrence relation again. When $i=1$ we have $(g_0)_{3, 2n}=a_2$ and hence $\ord_t(g_0)_{3, 2n}\geq \ord_ta_2$. By the recurrence relation, we need to consider $\ord_t a_{2j}(g_0)_{2i-2j+1, 2n}$ for $1\leq j \leq i-1$, which satisfies $$\ord_t a_{2j}(g_0)_{2i-2j+1, 2n}\geq \ord_ta_{2j}+\ord_ta_{2i-2j}$$ by induction. Recall that we have $\ord_ta_{2\alpha}\geq \left\lceil \frac{2\alpha m}{mk+l}\right\rceil$ for $1\leq \alpha \leq n$. Then $$\ord_t a_{2j}(g_0)_{2i-2j+1, 2n}\geq \ord_ta_{2j}+\ord_ta_{2i-2j}\geq \ord_t a_{2i}$$ and hence $\ord_t(g_0)_{2i+1, 2n}\geq \ord_ta_{2i}$.
\end{proof}

As a consequence, we define $g:=P^\vee\tilde{g}P/t^{m-1}$, then 

\begin{lemma}\label{lem:modified pairing C}
    The pairing $g$ is integral, non-degenerate, and skew-symmetric, and it satisfies $g(P^{-1}\gamma P)=-(P^{-1}\gamma P)^\vee g$.
\end{lemma}
\begin{proof}
    To show that $g$ is integral, we only need to show that:
    \begin{equation}\label{eq:integrality of pairing}
         p_i+p_j+\ord_t a_{i+j-2n-1}\ge m-1, i+j>2n+1.
    \end{equation}
    We may assume that $i\ge j$. Hence $p_i= \left\lceil\tfrac{(2n+1-i)m}{2n}\right\rceil-1$. If $1\le j\le n$, \eqref{eq:integrality of pairing} is due to the following
    \[
        \left\lceil\frac{(1-i)m}{2n}\right\rceil-\left\lceil\frac{jm}{2n}\right\rceil+\left\lceil\frac{(i+j-1)m}{2n}\right\rceil\ge 0.
    \]
    If $n+1\le j\le 2n$,  \eqref{eq:integrality of pairing} is due to the following
    \[
        \left\lceil\frac{(1-i)m}{2n}\right\rceil+\left\lceil\frac{(1-j)m}{2n}\right\rceil+\left\lceil\frac{(i+j-1)m}{2n}\right\rceil\ge 1.
    \]
    Since $g_{ij}=\pm 1$ for $i+j=2n+1$, $g$ is nondegenerate. The lastone follows naturally.
\end{proof}

\begin{proposition}\label{prop:C NP}
    The reduction of $\left(\gamma=\operatorname{diag} \{P_i^{-1}\gamma_i P_i\}_{i=1}^r,\; g=\operatorname{diag} \{g_i\}_{i=1}^r\right)$ gives the unique type C minimal reduction of $\gamma_0\in L^{\heartsuit}\mathfrak{sp}_{2n}$. 
\end{proposition}

\begin{proof}
    Here Lemma \ref{lem:Levi} does not apply. But we already see that the reduction of $\gamma$ is minimal and unique in type A, now we have constructed a non-degenerate, skew-symmetric $g$ compatible with $\gamma$, then the reduction must be minimal and unique in type C, see Subsection \ref{subsec:preparation modify pairing}.
\end{proof}

\section{Construction of Minimal Reduction: Types B and D}\label{type BD}

We now turn to the remaining—and most delicate—cases among the classical Lie algebras, namely the special orthogonal Lie algebras $\g = \mathfrak{so}_{2n}$ and $\g = \mathfrak{so}_{2n+1}$. 

Given $\gamma_0\in L^{\heartsuit}\mathfrak{g}$, then $\chi_{\gamma}(\lambda)$ is a self-dual polynomial of degree $2n$. Drawing the Newton polygon of $\chi_{\gamma_0}(\lambda)$ yields a factorization
\begin{align}\label{eq:chi decomposition}
    \chi_{\gamma_0}(\lambda) = \prod_i f_i(\lambda),
\end{align}
where each factor $f_i(\lambda)$ has Newton polygon of a single slope $\mu_i$, and these slopes are strictly decreasing:
\[
    \mu_1 > \mu_2 > \cdots.
\]
Each $f_i(\lambda)$ is self-dual. Let $2n_i$ denote the degree of $f_i(\lambda)$, and let $m_i$ be the $t$-adic valuation of the constant term $f_i(0)$.

As in the preceding sections, for each factor $f_i(\lambda)$ we construct a matrix $\gamma_i$ with characteristic polynomial $f_i(\lambda)$ whose type~A reduction is balanced. Two new difficulties arise in types~B and~D.

\smallskip
\noindent
\textbf{(i) Admissibility of the partition.}
Collecting the balanced partitions associated with all $\gamma_i$ (and, in type~B, additionally adjoining the part $[1]$) produces a partition, denoted $\mathbf d_A$. In general, $\mathbf d_A$ need not be admissible for types~B or~D (in the sense of the parity conditions on even parts).

\smallskip
\noindent
\textbf{(ii) Existence of compatible non-degenerate forms.}
Unlike the symplectic (type~C) situation, constructing for each $\gamma_i$ a compatible non-degenerate symmetric bilinear form can be subtle; in certain cases, no such form exists on a single factor.

\medskip
This section is organized as follows.

\begin{itemize}
\item[\textbf{Step 1.}]
We show that there exists a unique \emph{admissible} partition of type~B or~D (see Definition~\ref{def:admissible}) which is minimal among those admissible partitions that dominate $\mathbf d_A$. Moreover, if this partition occurs as a reduction of $\gamma_0$, then the minimal reduction of $\gamma_0$ is unique and its associated partition is precisely the one obtained in this way.

\item[\textbf{Step 2.}]
For each factor $f_i(\lambda)$, we construct a pair $(\gamma_i,g_i)$ such that the type~A reduction of $\gamma_i$ is balanced and $\ord_t(\det g_i)=m_i$. When $m_i\equiv 0$, we further conjugate $(\gamma_i,g_i)$ to $(\gamma_i',g_i')$ so that the reduction of $\gamma_i'$ matches the prediction of Step~1 and $g_i'$ is non-degenerate.

\item[\textbf{Step 3.}]
When $m_i\equiv 1$, we group together consecutive factors $\prod_{j=i}^{i+r} f_j$ (as in the Type~(D2) case in \S\ref{subsec:modify partition}), where $r$ is chosen so that the two ``odd'' blocks are paired.\footnote{Here $r$ is the smallest integer such that $m_{i+r}\equiv 1$ and $m_{i+1},\ldots,m_{i+r-1}\equiv 0$.} Then, in \S\ref{subsec:modify pairing}, we conjugate the pair $\left(\bigoplus_{j=i}^{i+r}\gamma_j,\ \bigoplus_{j=i}^{i+r}g_j\right)$ to a new pair $(\gamma,g)$ such that the reduction of $\gamma$ agrees with the partition predicted in Step~1 and $g$ is non-degenerate.
\end{itemize}

Consequently, we conclude the proof of the main theorem.

\begin{proof}[Proof of Theorem~\ref{thm.main}]
The cases of types~A and~C follow from Propositions~\ref{prop:A} and~\ref{prop:C}
(equivalently, Propositions~\ref{prop:minimal is balanced} and~\ref{prop:C NP}). For cases of types~B and~D, assuming that $\chi_{\gamma}(\lambda)$ decomposes as in~\eqref{eq:chi decomposition}, then we constructed partitions $\bf{d}_B$ and $\bf{d}_D$ as in ($\star\star$) and ($\star$) in \S\ref{subsec:modify partition}. Proposition~\ref{prop.type D minimal reduction}, Corollary~\ref{very even} and Proposition~\ref{prop.type B minimal reduction} shows that if $\bf{d}_B$ or $\bf{d}_D$ is a reduction of $\gamma$, then it must be minimal and unique. Then in Propositions~\ref{prop:B} and~\ref{prop:D}, we show that $\bf{d}_B$ and $\bf{d}_D$ can be realized as redution of $\gamma$.
\end{proof}

\subsection{Finding the Minimal Reduction}\label{find the reduction}

In this subsection, we explain how to modify the partition $\mathbf d_A$ into an admissible partition of type~B or~D, with the property that if it occurs as a reduction of $\gamma_0$, then it is the unique minimal reduction of $\gamma_0$.

We emphasize, however, that the minimal reduction depends not only on the valuations of the coefficients of $\chi_{\gamma_0}(\lambda)$, but also on certain relations among these coefficients. This subtle dependence shows that the Newton polygon alone is not always sufficient to determine the minimal reduction.

\subsubsection{Admissible partitions}

Following \cite{Spa88}, we introduce the following notation.

\begin{definition}\label{associated polynomial}   
    Let $d \in \bf{d}$ be a part of the partition such that $d \equiv 0$ and
    \[
        d_{i-1} > d_i = \cdots = d_{i+j-1} = d > d_{i+j}
    \]
    for some $i$ and $j$. Denote $d_{< i} := \sum_{\alpha < i} d_\alpha$.
    Then we define the associated graded polynomial
    \[
        \operatorname{gr}_{\bf{d}, d}(f) := \sum_{\alpha = 0}^{j} \left(a_{d_{< i} + \alpha d}\right)_{i + \alpha} X^{j - \alpha} \in \mathbb{C}[X].
    \]
\end{definition}

\begin{definition}\label{def:admissible}
    Consider $\gamma\in L^+\mathfrak{so}_{2n}$, then a partition $\bf{d}$ of $2n$ is said to be \emph{admissible} to $\gamma$, or equivalently, admissible to $\chi_{\gamma}(\lambda)$, if for any 
    \[
    d_{2i} > d_{2i+1} = d_{2i+2} = \cdots = d_{2i+2j}=d > d_{2i+2j+1}.
    \]
    with $d\equiv 0$, the polynomial $\operatorname{gr}_{\bf{d}, d}$ is a square in $\mathbb{C}[X]$. We can define the same notion in type B case similarly.
\end{definition}

\begin{proposition}{\cite[Proposition 6.4]{Spa88}}\footnote{In Claim 5 of the proof, one should read $y \in x + \mathfrak{m} \mathfrak{o}_N(\sO)$ instead.} \label{relation for type BD1*}
Let $\gamma \in L^+ \g$ be an element in type B or D. Let the partition associated to $\gamma(0)$ be $\bf{d}$, then $\bf{d}$ is admissible to $\gamma$.
\end{proposition}

\subsubsection{Modifying the Partition}\label{subsec:modify partition}

For each factor $f_i(\lambda)$ in the decomposition \eqref{eq:chi decomposition}, suppose $f_i(\lambda)$ has degree $2n_i$, $\ord_t a_{2n_i} = m_i$, and write $2n_i = m_i k_i + l_i$ with $0 \leq l_i < m_i$.

Let $\bf{d}_i$ denote the type A minimal reduction partition associated to $f_i(\lambda)$, i.e., $\RT^{\mathrm{A}}_{\min}(f_i)$. Define:
\[
\bf{d}_A =
\begin{cases}
   \mathrm{Re}\left(\bf{d}_1, \bf{d}_2, \ldots \right) \cup [1], & \text{in type B}, \\
   \mathrm{Re}\left(\bf{d}_1, \bf{d}_2, \ldots \right), & \text{in type D}.
\end{cases}
\]
Here $\mathrm{Re}(\ldots)$ denotes rearranging the parts into non-increasing order. Note that the re-ordering only happens at the case when $\left( \bf{d}_i, \bf{d}_{i+1} \right)$ has the form
\[
    \left((k_i+1)^{l_i}, k_i^{m_i-l_i}, (k_{i+1}+1)^{l_{i+1}}, k_{i+1}^{m_{i+1}-l_{i+1}} \right)
\]
such that $k_i = k_{i+1}$.

We now describe how to modify the partition $\bf{d}_A$ to obtain a reduction in type D.

\vspace{0.3cm}
\paragraph{\textbf{$(\star)$ Construction in type D:}}
\begin{enumerate}
    \item[\textbf{(D1)}] If $m_1 \equiv 0$:
    \begin{itemize}
        \item[(D1.1)] If either $l_1 \neq 0$ or $k_1 \equiv 1$, then set $\widetilde{\bf{d}}_1 = \bf{d}_1 = [(k_1 + 1)^{l_1}, k_1^{m_1 - l_1}]$.
        
        \item[(D1.2)] If $l_1 = 0$ and $k_1 \equiv 0$, then $\bf{d}_1 = [k_1^{m_1}]$ is already of type D, we modify as follows:
        \begin{itemize}
            \item[(D1.2.1)] If $\gr_{\bf{d}_A,k_1}(\chi_{\gamma_0}(\lambda))$ is a square (in the sense of Proposition~\ref{relation for type BD1*}), set $\widetilde{\bf{d}}_1 = \bf{d}_1$;
            \item[(D1.2.2)] Otherwise, set $\widetilde{\bf{d}}_1 = [k_1 + 1, k_1^{m_1 - 2}, k_1 - 1]$.
        \end{itemize}
    \end{itemize}
    
    \item[\textbf{(D2)}] If $m_1 \equiv 1$:
    
    Let $r \geq 2$ be the smallest integer such that $m_2 \equiv \cdots \equiv m_{r-1} \equiv 0$ but $m_r \equiv 1$. Then we modify 
    \[
        (\bf{d}_1, \ldots, \bf{d}_r) \quad \text{to} \quad (\widetilde{\bf{d}}_1, \bf{d}_2, \ldots, \bf{d}_{r-1}, \widetilde{\bf{d}}_r),
    \]
    i.e., we keep $\bf{d}_i$ unchanged for $i = 2, \ldots, r - 1$, and apply modifications only to
    \[
         \bf{d}_1 = \left[(k_1+1)^{l_1}, k_1^{m_1-l_1}\right], \qquad \bf{d}_r = \left[(k_r+1)^{l_r}, k_r^{m_r-l_r}\right]
    \]
    resulting in:
    \begin{align*}
        \widetilde{\bf{d}}_1 = \left[(k_1+1)^{l_1+1}, k_1^{m_1-l_1-1}\right], \qquad
        \widetilde{\bf{d}}_r = 
        \begin{cases}
        \left[(k_r+1)^{l_r-1}, k_r^{m_r-l_r+1}\right], & l_r \neq 0; \\
        \left[ k_r^{m_r-l_r-1}, k_r-1\right], & l_r = 0.
        \end{cases}
    \end{align*}
\end{enumerate}
By re-ordering and induction, we obtain a partition $\bf{d}_D$ of type D.

\begin{proposition}\label{prop.type D minimal reduction}
    Let $\gamma_0 \in L^{\heartsuit}\mathfrak{so}_{2n}$, and let $\bf{d}_A := \RT_{\min}^A(\gamma_0)$ be the type~A minimal reduction of $\gamma_0$, viewing $\gamma_0$ as an element of $L^{\heartsuit}\mathfrak{sl}_{2n}$. Let $\bf{d}_D$ be the partition as in ($\star$) above. Then $\bf{d}_D$ is the unique minimal admissible partition of type~D dominating $\bf{d}_A$, i.e., $\bf{d}_D \ge \bf{d}_A$ and for any admissible partition $\mathbf d'_D$ of type~D with $\bf{d}'_D \ge \bf{d}_A$, one has $\bf{d}'_D \ge \bf{d}_D$.
\end{proposition}

Before giving the proof, we record the following useful lemmas concerning partitions of type D.

Let $\bf{d}_A=\mathrm{Re}(\bf{d}_1, \ldots, \bf{d}_r)$ satisfy the condition (D2) above. Denote by $\#\bf{d}_A$ the number of parts in $\bf{d}_A$. Let $\bf{d}'$ be a partition of type D. We put 0's in the end of $\bf{d}'$ such that $\#\bf{d}' = \#\bf{d}_A$, and we rewrite $\bf{d}'=[\bf{d}'_1, \ldots, \bf{d}'_r]$ where $\#\bf{d}'_i = \#\bf{d}_i$. Obviously, $\bf{d}'$ falls into one of the following two cases:
\begin{itemize}
    \item[ ] \textbf{Case 1:} $\bf{d}'_i = \bf{d}_i$ for all $i = 2, \ldots, r - 1$;
    \item[ ] \textbf{Case 2:} There exists $2 \leq j \leq r-1$ such that $\bf{d}'_j \neq \bf{d}_j$.
\end{itemize}

\begin{lemma}\label{lem:Case 1}
In Case 1 above, the minimal type D partition $\bf{d}_D$, that is admissible to $\gamma_0$, is unique. Moreover,
\[
    \bf{d}_D = \mathrm{Re}\left(\widetilde{\bf{d}}_1, {\bf{d}}_2, \ldots, {\bf{d}}_{r-1}, \widetilde{\bf{d}}_r \right),
\]
where each $(\widetilde{\bf{d}}_1, \widetilde{\bf{d}}_r)$ is defined as in condition (D2) above.
\end{lemma}
\begin{proof}
    We only need to deal with $\left( \bf{d}_1, \bf{d}_r \right) = \left( (k_1+1)^{l_1}, k_1^{m_1-l_1}, (k_r+1)^{l_r}, k_r^{m_r-l_r} \right)$. Since $m_1 \equiv m_r \equiv 1$, we have the following cases:
    \begin{enumerate}
        \item $k_1, k_r \equiv 0$, then $l_1, l_r \equiv 0$;
        \item $k_1, k_r \equiv 1$, then $l_1, l_r \equiv 1$;
        \item $k_1 \equiv 1$, $k_r \equiv 0$, then $l_1 \equiv 1$, $l_r \equiv 0$;
        \item $k_1 \equiv 0$, $k_r \equiv 1$, then $l_1 \equiv 0$, $l_r \equiv 1$.
    \end{enumerate}
    We only prove for case (1), the other cases follow similarly.

    First, we assume $k_1 \neq k_r$, let $\bf{d}'=[d'_1, d'_2, \ldots]$ be a type D partition that $\bf{d}' > [\bf{d}_1, \bf{d}_r]$. Then $\sum_{i=1}^{m_1} d'_i \geq m_1k_1+l_1=2n_1$. 

    We claim that equality cannot hold. Otherwise, 
    \[
        [d'_1,\ldots d'_{2n_1}] = [(k_1+1)^{l_1}, k_1^{m_1-l_1}].
    \]
    Since $\bf{d}'$ is type D, we have $d'_{2n_1+1} = \ldots = d'_{2n_1+q} = k_1$ for some $q \equiv 1$. We may compute $\operatorname{gr}_{\bf{d}', k_1}$ in this case, which is a polynomial of odd degree, so this cannot happen. Then $\sum_{i=1}^{m_1} d'_i > 2n_1$, since $\bf{d}'$ is type D, $[d'_1, \ldots, d'_{2n_1}] \geq \widetilde{\bf{d}}_1$, and by the construction of $\widetilde{\bf{d}}_r$, we conclude that $\bf{d}' \geq \bf{d}_D=[\widetilde{\bf{d}}_1, \widetilde{\bf{d}}_r]$.

    If $k_1 = k_r$, in this case, $r=2$. Then
    \[
        \mathrm{Re}\left( (k_1+1)^{l_1}, k_1^{m_1-l_1}, (k_r+r)^{l_r}, k_r^{m_r-l_r} \right) = [(k_1+1)^{l_1+l_r}, k_1^{m_1+m_r-l_1-l_r}]
    \]
    which is automatically type D. and the computation of $\operatorname{gr}_{\bf{d}_D, k_1}$ show that it is a polynomial of degree 0, thus, it is admissible to $\gamma_0$. Moreover, it is easy to see that it coincides with $\mathrm{Re}\left( \widetilde{\bf{d}}_1, \widetilde{\bf{d}}_2 \right)$.
\end{proof}

\begin{lemma}\label{lem:Case 2}
    In Case 2, the type D partition $\bf{d}'$ is either strictly larger than the minimal partition from Case 1, or there exists a block
    \[
        d'_{2i_0} > d'_{2i_0+1} = d'_{2i_0+2} = \cdots = d'_{2i_1} > d'_{2i_1+1}, \quad \text{with } d'_{2i_1} \equiv 0,
    \]
    lying in the interval $[\bf{d}'_j, \bf{d}'_{j+1}]$ for some $j$, but not contained in $[\bf{d}_j, \bf{d}_{j+1}]$. Moreover, there exists a positive odd integer $k$ such that
    \[
        d'_{2i_0+1}, \ldots, d'_{2i_0 + k} \in \bf{d}'_j, \qquad d'_{2i_0 + k + 1}, \ldots, d'_{2i_1} \in \bf{d}'_{j+1}.
    \]
\end{lemma}
\begin{proof}
    Similarly as in the proof of Lemma~\ref{lem:Case 1}, we have the following cases:
    \begin{enumerate}
        \item $k_1, k_r \equiv 0$, then $l_1, l_r \equiv 0$;
        \item $k_1, k_r \equiv 1$, then $l_1, l_r \equiv 1$;
        \item $k_1 \equiv 1$, $k_r \equiv 0$, then $l_1 \equiv 1$, $l_r \equiv 0$;
        \item $k_1 \equiv 0$, $k_r \equiv 1$, then $l_1 \equiv 0$, $l_r \equiv 1$.
    \end{enumerate}
    We only prove for case (1), the other cases follow similarly.
    
    Let $\bf{d}'=[d'_1,d'_2, \ldots]$, $\sum_{i=1}^{m_1} d'_i \geq 2n_i$, if the equality holds, then we have $d''_{2n_1+1} = \ldots = d''_{2n_1+q} = k_1$ for some $q \equiv 1$. Otherwise, $\sum_{i=1}^{m_1} d'_i \geq \widetilde{\bf{d}}_1$, moreover, since $\widetilde{\bf{d}}_r < \bf{d}_r$, we conclude $\bf{d}' \geq \bf{d}_D$.   
\end{proof}

\begin{proof}[Proof of Proposition~\ref{prop.type D minimal reduction}]
    Without loss of generality, we assume there is no 
    \[
        \left( \bf{d}_i, \bf{d}_{i+1} \right) = \left( (k_i+1)^{l_i}, k_i^{m_i-l_i}, (k_{i+1}+1)^{l_{i+1}}, k_{i+1}^{m_{i+1}-l_{i+1}} \right)
    \]
    such that $k_i = k_{i+1}$. Otherwise, from the proof of Lemma~\ref{lem:Case 1}, we replace $\left( \bf{d}_i, \bf{d}_{i+1} \right)$ by 
    \[
        \left( (k_i+1)^{l_i+l_{i+1}}, k_i^{m_i+m_{i+1}-l_i-l_{i+1}} \right). 
    \]
    Then, $\bf{d}_A=[\bf{d}_1,\bf{d}_2, \ldots]$ and $\bf{d}_D=[ \widetilde{\bf{d}}_1, \widetilde{\bf{d}}_2, \ldots ]$.

    Let $\bf{d}'=[d'_1, d'_2, \ldots]=[\bf{d}'_1, \bf{d}'_2, \ldots]$ be an admissible type D partition that larger than $\bf{d}_A$, and we put 0's in $\bf{d}'$ such that $\#\bf{d}'_i=\#\bf{d}_i$. If $\bf{d}_1$ satisfies condition (D1.2.2), i.e., $\gr_{\bf{d}_A, k_1}(\chi_{\gamma_0}(\lambda))$ is not a square, then there exists $1 \leq j \leq m_1$ such that $\sum_{i=1}^{j} d'_i > \sum_{i=1}^{j} k_1$. Since $\bf{d}'$ is type D and $\bf{d}_1=[k_1^{m_1}]$, then $j =1$ and $\bf{d}'_1 \geq \widetilde{\bf{d}}_1$. For other conditions (D1.1) and (D1.2.1), $\widetilde{\bf{d}}_1 = \bf{d}_1$, then $\bf{d}'_1 \geq \widetilde{\bf{d}}_1$.

    If $\bf{d}_1$ satisfies condition (D2), it has been delt in Lemma~\ref{lem:Case 1},~\ref{lem:Case 2}. Note that the second case in Lemma~\ref{lem:Case 2} cannot happen since $\operatorname{gr}_{\bf{d}', d'_{2i_1}}$ is a polynomial of odd degree, which can not be a square. Thus, $\bf{d}'_1 \geq \widetilde{\bf{d}}_1$. By induction, we conclude.
\end{proof}

In type D, two very even orbits which are conjugate under $\operatorname{O}_{2n}$ will associate to the same partition, thus in this case, it is not enough to determine the minimal reduction is unique by showing that the partition of the minimal reductions are unique. However, the following lemma will adjust this case.

Let $\bf{O}^i$, $i= \mathrm{I}, \mathrm{II}$ be two very even orbits which are conjugate by $\operatorname{O}_{2n}$, and we use $\bf{d}$ to denote their common partition. We consider $\left(L^+{\mathfrak{so}_{2n}}\right)_{\bf{O}^i}$ defined as before and denote the common image of $\left(L^+\mathfrak{so}_{2n}\right)_{\bf{O}^i}$ in $L^+\mathfrak{c}$ to be $L^+\mathfrak{c}_{\bf{d}}$. Notice that $\left(L^+\mathfrak{so}_{2n}\right)_{\bf{O}^i}$ and $L^+\mathfrak{c}_{\bf{O}}$ are locally closed.

\begin{lemma}
    For any $\gamma\in \left(L^+\mathfrak{so}_{2n}\right)_{\bf{O}^\mathrm{I}}$, and $g\in L\operatorname{SO}_{2n}$, then $$\Ad_{g^{-1}}(\gamma)\notin \left(L^+\mathfrak{so}_{2n}\right)_{\bf{O}^\mathrm{II}}.$$
\end{lemma}

\begin{proof}
    This proof is a local analogue of Proposition 5.3 (case (4) in the proof) of \cite{BK18} or Proposition 4.24 of \cite{WWW25}. In particular, if we regard element $f$ in $L^+\mathfrak{c}_{\bf{d}}$ as characteristic polynomials, then for any $d\in \bf{d}$, the associate polynomial $\operatorname{gr}_{\bf{d},d}$, defined in Definition \ref{associated polynomial}, by Proposition~\ref{relation for type BD1*}, is a square. Now since $\bf{d}$ is a very even partition, let $d_1=\max\{d\in \bf{d}\}$ and $d_r=\min\{d\in \bf{d}\mid d\neq 0\}$, we see that $\operatorname{gr}_{\bf{d},d_1}$ is a monic polynomial and the constant term of $\operatorname{gr}_{\bf{d},d_r}$ is given by the square of the degree $\frac{\mid \bf{d} \mid}{2}$ term of the Pfaffian. While if $d_i>d_{i+1}\in \bf{d}$, we see that the constant term of $\operatorname{gr}_{\bf{d},d_i}$ equals to the coefficient of leading term of $\operatorname{gr}_{\bf{d},d_{i+1}}$. By a similar argument as in Proposition 5.3 (case (4) in the proof) of \cite{BK18} or Proposition 4.25 of \cite{WWW25}, we see that the map 
    \[
        \left(L^+\mathfrak{so}_{2n}\right)_{\bf{O}^\mathrm{I}}\bigsqcup  \left(L^+\mathfrak{so}_{2n}\right)_{\bf{O}^\mathrm{II}}\longrightarrow L^+\mathfrak{c}_{\bf{d}}
    \] factors through a cover of $L^+\mathfrak{c}_{\bf{d}}$, which is a disjoint union of two isomorphism spaces. Moreover, this map sends $\left(L^+\mathfrak{so}_{2n}\right)_{\bf{O}^\mathrm{I}}$ or $\left(L^+\mathfrak{so}_{2n}\right)_{\bf{O}^\mathrm{II}}$ to different components. However, an adjoint action of $g$ on $\gamma$ would not change its image in $L^+\mathfrak{c}_{\bf{d}}$ as well as its image in the cover, which means $\Ad_{g^{-1}}(\gamma)\notin \left(L^+\mathfrak{so}_{2n}\right)_{\bf{O}^\mathrm{II}}$.
\end{proof}

\begin{corollary}\label{very even}
    For any $\gamma\in L^{\heartsuit}\mathfrak{so}_{2n}$, if $\RT^{\mathrm{A}}_{\min}(\gamma)$ is a nilpotent orbit whose partition is very even, then $\operatorname{RT}_{\min}(\gamma)$ is a singleton.
\end{corollary}

\vspace{0.3cm}
\paragraph{\textbf{$(\star\star)$ Construction in type B:}}

Suppose the constant term of $\chi_{\gamma}(\lambda)$ has order $m$. Recall that in the type B case, the type A minimal reduction is given by
\[
    \bf{d}_A = \mathrm{Re}\left(\bf{d}_1, \bf{d}_2, \ldots \right) \cup [1].
\]
We define the modified partition $\bf{d}_B$ as follows:

\begin{enumerate}
    \item[\textbf{(B1)}] If $m\equiv 0$, define
    \[
        \bf{d}_B := \left[ \bf{d}_D, 1 \right],
    \]
    where $\mathbf d_D$ is obtained from $(\mathbf d_1,\mathbf d_2,\ldots)$ by applying the same modification rules as in the type~D case.

    \item[\textbf{(B2)}] If $m\equiv 1$, suppose that there exists an index $r>1$ such that $m_r \equiv 1$ and $m_j \equiv 0$ (then, $l_j \equiv 0$) for all $j > r$. We then modify
    \[
        \left(\bf{d}_r, \bf{d}_{r+1}, \ldots \right) \cup [1]
    \]
    to be
    \[
        \left[\widetilde{\bf{d}}_r, \bf{d}_{r+1}, \ldots \right]
    \]
    where
    \[
        \widetilde{\bf{d}}_r=\left[(k_r+1)^{l_r+1}, k_r^{m_r-l_r-1}\right]
    \]
    The remaining part $\left(\bf{d}_1, \ldots, \bf{d}_{r-1} \right)$ is modified in the same manner as in the type~D case.
\end{enumerate}

\begin{proposition}\label{prop.type B minimal reduction}
With the notation as above, the $\bf{d}_B$ is the unique minimal admissible partition of type~B dominating $\bf{d}_A$, i.e., $\bf{d}_B \ge \bf{d}_A$ and for any admissible partition $\mathbf d'_B$ of type~D with $\bf{d}'_B \ge \bf{d}_A$, one has $\bf{d}'_B \ge \bf{d}_B$.
\end{proposition}
\begin{proof}
The argument is analogous to the proof of Proposition~\ref{prop.type D minimal reduction}.
\end{proof}

\subsection{Newton Polygon of Single Slope}\label{single slope}

Assume that
\[
    f(\lambda) = \lambda^{2n} + a_2 \lambda^{2n - 2} + \cdots + a_{2n - 2} \lambda^2 + a_{2n},
\]
with $\ord_t a_{2n} = m > 0$ and 
\[
    \ord_t a_{2i} \geq \left\lceil 2i \cdot \frac{m}{2n} \right\rceil \quad \text{for all} \quad 1 \leq i \leq n. 
\]
Then the Newton polygon of $f(\lambda)$ consists of a single segment.

We first consider the case where $\gamma_0 = \operatorname{R}(f(\lambda))$ is the companion matrix of $f(\lambda)$, acting on the module
\[
    \mathcal{E}_0 := \frac{\sO[\lambda]}{(f(\lambda))},
\]
with basis $\{1, \lambda, \ldots, \lambda^{2n-1}\}$.

Define an involution $\sigma$ on $\mathcal{E}_0$ by $\sigma(\lambda^i) = (-1)^i \lambda^i$. Then we define a pairing $g_0$ on $\mathcal{E}_0$ by
\[
    (g_0)_{i,j} := g_0(\lambda^{i-1}, \lambda^{j-1}) := \operatorname{Tr}\left( \frac{\lambda \cdot \lambda^{i-1} \cdot \sigma(\lambda^{j-1})}{f'(\lambda)} \right).
\]
This defines a symmetric bilinear pairing $g_0$ on $\mathcal{E}_0$ such that $g_0 \gamma_0$ is skew-symmetric. However, note that $g_0$ is not perfect over $\sO$.

\begin{lemma}\label{order of g_0}
    $\det g_0=a_{2n}$ and $\ord_t (g_0)_{i,j}\geq \ord_ta_{i+j-2n}$ for $i+j\geq 2n+1$.
\end{lemma}

\begin{proof}
    To compute the determinant of $g_0$, it suffice to notice that $g_0\operatorname{R}(f)^{-1}$ satisfies the following condition: the $(i,j)$ entry is $0$ if $i+j\leq 2n$ and the $(i, 2n-i)$ entry is $(-1)^{i+1}$. And the computation of orders of $(g_0)_{i,j}$ is similar to Lemma \ref{pairing g in skew case}.
\end{proof}    

\begin{definition}
    A symmetric matrix $g\in M_{n\times n}(\sO)$ is said to be \emph{optimal} with respect to $I\subset \{1, \cdots , n\}$ if for any $i\in I$, the entries in $i$-th row of $g$ lies in $t\sO$ and if we multiply $t^{-1}$ to each $i$-th row for all $i\in I$, the resulting matrix is invertible over $\sO$.

    Consequently, if $g$ is optimal with respect to $I$, using the standard basis $\{e_1, \cdots , e_n\}$ of $\sO^{\oplus n}$, $g$ induces an exact sequence $$0\longrightarrow \sO^{\oplus n}\stackrel{g}{\longrightarrow}(\sO^{\oplus n})^\vee\longrightarrow Q \longrightarrow 0 $$ where $Q$ is a vector space over $\mathbb{C}$, with basis $\{\overline{e_i^\vee}\}_{i\in I}$.
\end{definition}

We write $2n=mk+l$ so that $0\leq l <m$. The associated balanced partition is $[(k+1)^l, k^{m-l}]$. Since we are working in type B or D, even parts in the partition must occur with even multiplicity. Accordingly, we distinguish the three cases:

\begin{itemize}
    \item  $m\leq 2n$ and $l=0$;
    \item  $m\leq 2n$ and $l>0$;
    \item  $m> 2n$.
\end{itemize}

\subsubsection{The cases that $m\leq 2n$ and $l=0$}  Define 
\[
    P=\mathrm{diag}\left( \{t^{m-1}\}^{k}, \cdots \{t^{i}\}^{k}, \cdots , \{1\}^{k} \right) 
\]  
 Then we have a submodule $\mathcal{E}=P^{-1}\mathcal{E}_0$ on which we have a linear transform $\gamma$ and a symmetric pairing $g$ where $$\gamma=P^{-1}\gamma_0P\ \ \text{and}\ \ g=P^{\vee}g_0P/t^{m-1}.$$

Since $\ord_t a_{2i}\ge \left\lceil 2i\cdot \frac{m}{2n} \right\rceil$, $\gamma$ is integral, i.e., $\gamma\in L^{+}\g$. By Lemma~\ref{order of g_0}, $g$ is integral and optimal with respect to $\{ik\}_{1\leq i \leq m}$, hence it defines an exact sequence $$0\longrightarrow \mathcal{E}\stackrel{g}{\longrightarrow}\mathcal{E}^\vee\longrightarrow Q \longrightarrow 0$$ here $Q$ is a vector space of dimension $m$ over $\mathbb{C}$. 

We define the following polynomial
\[
    \operatorname{gr}_P(f)=X^m+\sum_{i=1}^m(a_{ik})_iX^{m-i}\in \mathbb{C}[X],
\]
where $(a_{ik})_i$ denotes the coefficient of $t^i$ in $a_{ik}$. Notice that if $k\equiv 1$, then $\operatorname{gr}_P(f)$ is self dual.

The pairing $g$ induces a symmetric pairing $g_Q$ on $Q$, which is non-degenerate by Lemma~\ref{order of g_0}.  For a subspace $W$ of $Q$, we consider its preimage $\mathcal{E}_W$ in $\mathcal{E}^\vee$ and the induced map $\gamma_W:\sE_W\rightarrow \sE_W$. Our goal is to compute the reduction of $\gamma_W$ and ensure that there is an induced non-degenerate pairing on $\sE_W$.

The computation of reduction of $\gamma_W$ follows closely the discussion from Lemma~2.38 to Proposition~2.43 in~\cite{WWW24}. Firstly, the basis $\{1,\lambda,\ldots,\lambda^{2n-1}\}$ of $\sE_0$ induces a base $\{e_1, \cdots , e_{2n}\}$ of $\sE$ and hence the dual basis $\{e_1^\vee, \cdots ,e_{2n}^\vee\}$ of $\sE^\vee$. With respect to this basis, the matrix $\gamma^\vee$ acts on $\sE^\vee$ and we may regard $Q=\left\langle e_{ik}^\vee(0)\right\rangle_{1\leq i \leq m}$ as a subspace of $\sE^\vee(0)=\sE^\vee\otimes \mathbb{C}$. We denote $\sE_W=P_W\sE^\vee$ for some $P_W\in \operatorname{GL_{2n}}(\sK)$. Thus under the basis $P_W\{e_1^\vee, \cdots ,e_{2n}^\vee\}$, $\gamma_W$ is given by $P_W^{-1}\gamma^\vee P_W$. Let $\sT=\operatorname{diag}\left(\{\{1\}^{k-1}, t\}^{m}\right)$. Then $P_W^{-1}\sT$ is integral and $\tilde{\gamma^\vee}:=\sT^{-1}\gamma^\vee$ is also integral. Moreover, $\tilde{\gamma^\vee}$ is invertible.

From now on, we do not distinguish an integral matrix $A\in \operatorname{Mat}(\sO)$ from its reduction $A \operatorname{mod} t$. In particular, when we refer to the \emph{rank}, \emph{kernel}, or \emph{image} of $A$, we always mean those of $A \bmod t$.

Thus $\operatorname{rk}P_W=2n-\dim W$, $\Ker P_W^{-1}\sT=W$ and the partition of $\gamma^\vee$ is $[k^m]$. The calculation in the proof of Lemma 2.40 in \cite{WWW24} gives the following: 
\begin{lemma}\label{lem:parition for submodule}
   for $i \geq 1$, we have 
\begin{align*}
    \operatorname{rk}(\gamma_W)^i = \operatorname{rk}(\gamma^\vee)^i +\dim W-\dim\left(W\cap \Ker (\gamma^\vee)^{i-1}\right) -\dim \left(W\cap \left(\tilde{\gamma^\vee}(\Im(\gamma^\vee)^{i})\oplus \tilde{\gamma^\vee}((\gamma^\vee)^{i-1}W) \right) \right)
\end{align*} 
\end{lemma}

From the explicit form of $\gamma^\vee$, we see that 
\begin{lemma}\label{lem:intersection with W}
    \begin{enumerate}
        \item for $1\leq i \leq k$, $W\cap\Ker(\gamma^\vee)^{i-1}=0$.
        \item  for $1\leq i \leq k-1$, $W\subset Q \subset \tilde{\gamma^\vee}(\Im(\gamma^\vee)^{i})$.
        \item for $i=k$, $\Im (\gamma^\vee)^k=0$ and $\tilde{\gamma^\vee}((\gamma^\vee)^{k-1}W)=\operatorname{R}(\operatorname{gr}_P(f))^\vee W$. 
    \end{enumerate}
\end{lemma}
\begin{proof}
    The first is due to that $Q\cap\Ker(\gamma^\vee)^{i-1}=0$, for $1\leq i \leq k$.

    The second is due to that for $1\leq i \leq k-1$, we have 
\[
    \{e_{jk+1}^\vee(0)\}_{0\leq j \leq m-1}\subset \Im (\gamma^\vee)^{i}
\]

    The last one is due to that the characteristic polynomial of $\tilde{\gamma^\vee}((\gamma^\vee)^{k-1}W)$ mod $t$ is $\gr_{P}(f)$, i.e., the terms on the boundary of the Newton polygon.
\end{proof}

Therefore, in order for the partition of $\gamma_W$ to be $[k^m]$, it suffices to find a maximal isotropic subspace $W$ of $Q$ such that $W\cap \operatorname{R}(\operatorname{gr}_P(f))^\vee W=W$. And in order for the partition of $\gamma_W$ to be $[k+1,k^{m-2},k-1]$, it suffices to find a maximal isotropic subspace $W$ of $Q$ such that $\dim W -\dim W\cap \operatorname{R}(\operatorname{gr}_P(f)^\vee W=1$.

There is an induced perfect pairing on $\sE_W$ if and only if $W$ is isotropic under the pairing $g_Q$ on $Q$. Moreover, if $W$ is maximal isotropic, then the induced pairing is non-degenerate. Now we come to describe the pairing $g_Q$ on $Q$. It is more convenient to describe its inverse pairing $g_Q^{-1}$ on $Q^\vee$. Actually, under the basis chosen for $Q$ (hence the dual basis for $Q^\vee$), 
\[
    (g_Q^{-1})_{i,j}=(g_{ik,jk})_1=\left(\operatorname{Tr}\left(\frac{\lambda \lambda^{ik-1}\sigma(\lambda^{jk-1})}{f^\prime(\lambda)}\right) \right)_{i+j-m}.
\]

By the inequality $\ord_{t}a_{2i}\geq \left\lceil 2i\frac{m}{2n} \right\rceil$ and the calculation of trace, we see that $(g_Q^{-1})_{i,j}$ is exactly the coefficient of $X^{m-1}$ of the expression of $(-1)^{j}X^{i+j}$ in the ring $\mathbb{C}[X]/(\operatorname{gr}_P(f))$. Consequently, we obtain the relations that if $k\equiv 1$, then  
\[
    g_{Q}^{-1}\operatorname{R}(\operatorname{gr}_P(f))=-\operatorname{R}(\operatorname{gr}_P(f))^\vee g_{Q}^{-1}.
\]
and if $k\equiv 0$, then  
\[
    g_{Q}^{-1}\operatorname{R}(\operatorname{gr}_P(f))=\operatorname{R}(\operatorname{gr}_P(f))^\vee g_{Q}^{-1}.
\]

\begin{proposition}\label{m even and l=0}
With the above notation, we obtain 
\begin{enumerate}
    \item Suppose that $m \equiv 0$ and $k \equiv 1$.  
    Then there exists a choice of $W$ such that the reduction of $\gamma_W$ has partition $[k^m]$, and the induced pairing on $\mathcal{E}_W$ is non-degenerate.

    \item Suppose that $m \equiv 0$ and $k \equiv 0$.  
    Then there exists a choice of $W$ such that the reduction of $\gamma_W$ has partition $[k^m]$ and the induced pairing on $\mathcal{E}_W$ is non-degenerate if and only if $\operatorname{gr}_P(f)$ is a square in $\mathbb{C}[X]$.
    
    If $\operatorname{gr}_P(f)$ is not a square, then there exists a choice of $W$ such that the reduction of $\gamma_W$ has partition $[\,k+1,\, k^{m-2},\, k-1\,]$, and the induced pairing on $\mathcal{E}_W$ is non-degenerate.
\end{enumerate}
\end{proposition}
\begin{proof}
    By discussion above, if $m\equiv 0$ and $k\equiv 1$, then $\operatorname{R}(\operatorname{gr}_P(f))$ is anti-self-dual with respect to $g_{Q}^{-1}$, so by Lemma \ref{self dual polynomial}, we have a maximal isotropic subspace ${}^*W$ of $Q^\vee$ so that $\operatorname{R}(\operatorname{gr}_P(f))({}^* W)={}^*W$. Let $W\subset Q$ be the orthogonal complement of ${}^*W$ with respect to the natural pairing between $Q$ and $Q^\vee$. Then $\operatorname{R}(\operatorname{gr}_P(f))^\vee(W)=W$, and therefore the reduction of $\gamma_W$ is $[k^m]$.

    If $m\equiv 0$ and $k\equiv 0$, then $\operatorname{R}(\operatorname{gr}_P(f))$ is self-dual with respect to $g_{Q}^{-1}$, so by Lemma \ref{square of polynomial}, we have a maximal isotropic subspace ${}^*W$ of $Q^\vee$ so that $\operatorname{R}(\operatorname{gr}_P(f))({}^* W)={}^*W$ if and only if $\operatorname{gr}_P(f)$ is a square. And if $\operatorname{gr}_P(f)$ is not a square, by Proposition \ref{prop:codim 1}, we have a maximal isotropic subspace ${}^*W$ of $Q^\vee$ so that $\codim_{{}^*W}\operatorname{R}(\operatorname{gr}_P(f))({}^* W)\cap {}^*W=1$. As before, let $W\subset Q$ be the orthogonal complement of ${}^*W$. Then the reduction of $\gamma_W$ is $[k^m]$ or $[k+1, k^{m-2}, k-1]$ depending on that $\operatorname{gr}_P(f)$ is a square or not.
\end{proof}

\subsubsection{The cases that $m\leq 2n$ and $l\neq 0$}  We take $\gamma_0$ and $g_0$ as before, and choose $P$ to be a diagonal matrix whose $j$-th diagonal entry is $t^{p_j}$, where
\[
    p_{j}=
    \begin{cases}
        m-\left\lceil\tfrac{jm}{2n}\right\rceil, & 1\le j\le n,\\
        \left\lceil\tfrac{(2n+1-j)m}{2n}\right\rceil-1, & n+1\le j\le 2n.
    \end{cases}
\]

\begin{lemma}\label{integral and optimal}
    $\gamma=P^{-1}\gamma_0P$ and $g=P^\vee g_0 P/t^{m-1}$ are integral. Moreover, the partition of $\gamma(0)$ is $[(k+1)^{l}, k^{m-l}]$ and $g$ is optimal with respect to 
    \[
        I_P=\left\{ik+\left\lfloor\frac{il}{m}\right\rfloor, jk+\left\lceil\frac{jl}{m}\right\rceil\right\}_{1\le i\le \left\lfloor\frac{m}{2}\right\rfloor, \left\lfloor\frac{m}{2}\right\rfloor+1\le j\le m}.
    \]
    We write $I_P=\{\alpha_1< \alpha_2< \cdots < \alpha_{m}\}$. In particular, $p_{\alpha_i}=m-i$.
\end{lemma}

\begin{proof}
    Since $\gamma_0$ is adjoint, and the construction of $P$ is the same as that for $\mathfrak{sp}_{2n}$, by Lemma \ref{lem:integrality of modified gamma C}, \ref{lem:partition of modified gamma C}, the statements for $\gamma$ holds in the same way.

    By \eqref{eq:integrality of pairing}, $\ord_t a_{i+j-2n}+p_i+p_j-(m-1)\ge 0$. Hence $g$ is integral and $\ord\det g=\ord\det g_0=m$. Notice that $g_{i,j}=0$ if $i+j< 2n$.  
    
    Now we determine when $\ord g_{i,j}=1$ if $i+j=2n$. We assume that $i\ge j$,
    \[
    \ord g_{i,j}=-\left\lceil\tfrac{jm}{2n}\right\rceil+\left\lceil\tfrac{(1-i)m}{2n}\right\rceil =1
    \]
    In other words, $\left\lceil\tfrac{(j+1)m}{2n}\right\rceil-\left\lceil\tfrac{jm}{2n}\right\rceil=1$, i.e. when $p_j$ jumps. Then the statement follows from Lemma \ref{lem:jumps} and that $g$ is symmetric.
\end{proof}

We construct $\mathcal{E}$ as before and choose the same type of basis. By the choice of $P$, in the present case we have $Q = \left\langle e_{\alpha_i}^\vee(0) \right\rangle_{1\leq i \leq m}$. Under the chosen basis of $Q$, we define a linear transform $\Phi$ on $Q$ using the matrix $\operatorname{R}(\operatorname{gr}_P(f))^\vee$ where the polynomial $\operatorname{gr}_P(f)$ is defined as $$\operatorname{gr}_P(f)=X^{m}+\sum_{i=1}^{m}(a_{\alpha_i})_iX^{m-i}.$$

We further decompose $Q$ as follows: Let $Q^{k+1} \subset Q$ be the subspace generated by 
\[
e_{\alpha_i}^\vee(0)
\quad \text{such that} \quad
\alpha_i - \alpha_{i-1} = k+1
\quad (\text{with } \alpha_0 := 0),
\]
and let $Q^{k} \subset Q$ be the subspace generated by the remaining basis elements.  
Then we have a direct sum decomposition
\[
Q = Q^{k+1} \oplus Q^{k}.
\]
As before, we regard $Q$ as a subspace of $\mathcal{E}(0)$.

Following the same strategy as in the previous case, for a subspace $W$ of $Q$, we consider its inverse image $\sE_W$. Our goal is to compute the partition of the induced endomorphism $\gamma_W$. To make the partition of $\gamma_W$ be $[(k+1)^{l}, k^{m-l}]$, it suffices to choose $W$ such that $\operatorname{rk}(\gamma_W)^i - \operatorname{rk}(\gamma^\vee)^i=0$ for all $i\geq 1$. As Lemma \ref{lem:intersection with W}, we have

\begin{lemma}
   $W\cap\Ker(\gamma^\vee)^{i-1}=0$ for $1\leq i \leq k$; and for $1\leq i \leq k-1$, $W\subset \tilde{\gamma^\vee}(\Im(\gamma^\vee)^{i})$.  
\end{lemma}
Moreover,
\begin{lemma}\label{lem:boundary terms}
    For $i=k$, we have 
\[
    W\cap \left(\tilde{\gamma^\vee}(\Im(\gamma^\vee)^{k})\oplus \tilde{\gamma^\vee}((\gamma^\vee)^{k-1}W) \right)=W\cap \left(\Phi(Q^{k+1}) \oplus \Phi(W\cap Q^{k})\right)
\]
\end{lemma}
\begin{proof}
    The action of $\gamma^\vee$ decrease the index of $e_j^\vee(0)$ by $1$ for $j$ such that $j-1\notin I_P$. By the explicit form of $\tilde{\gamma^\vee}$, one see that $\tilde{\gamma^\vee}(e_{\alpha_j+1}^\vee(0))=e_{\alpha_j}^\vee(0)+(a_{2n-\alpha_j})_{p_{\alpha_j+1}+1}e_{2n}^\vee(0)$. Notice that $p_{\alpha_j+1}+1=p_{\alpha_j}$, then by the definition of $\Phi$, we see that
    \[
    Q\cap \tilde{\gamma^\vee}(\Im(\gamma^\vee)^{k})=Q\cap \Phi(Q^{k+1})
    \]
    Similarly, we have 
    \[
    Q\cap \tilde{\gamma^\vee}((\gamma^\vee)^{k-1}W=Q\cap \Phi(W\cap Q^{k})
    \]
    So we have the equality.
\end{proof}
So this means that we need $W$ to satisfy 
\[
    W \subset \Phi(Q^{k+1})+\Phi(W\cap Q^{k}).
\]
Now we consider $i=k+1$, then $W\cap\Ker(\gamma^\vee)^{k}=W\cap Q^{k}$, $\Im (\gamma^\vee)^{k+1}=0$ and
\[
    \tilde{\gamma^\vee}((\gamma^\vee)^{k}W)=\Phi(\overline{(W)_{=k+1}})
\]
where $\overline{(W)_{=k+1}}$ denotes the projection of $W$ into $Q^{k+1}$. Consequently, we further require 
\[
    \dim W= \dim W\cap Q^{k}+ \dim W\cap \Phi(\overline{(W)_{=k+1}}).
\]

As before, to induce the pairing on $\sE_W$, we also need $W$ to be isotropic. So we need to describe the pairing $g_Q$ on $Q$. Since it is more easy to describe the pairing $g_Q^{-1}$ in $Q^\vee$, we first transfer the conditions for $W$ to the conditions for a certain subspace in $Q^\vee=(Q^{k+1})^\vee\oplus (Q^{k})^\vee$.

For a subspace $V$ in $Q$, we denote by $^*V \subset Q^\vee$ its orthogonal complement of $V$ with respect to the natural pairing between $Q$ and $Q^\vee$. Then the following identities hold:
\begin{itemize}
    \item  $^*(Q^{k+1}+W\cap Q^k)=(Q^{k})^\vee\cap {}^*(W\cap Q^k)$;
    \item  $^*(W\cap Q^k)={}^*W+(Q^{k+1})^\vee$;
    \item  $^*(\overline{(W)_{=k+1}})=(Q^{k})^\vee+{}^*W\cap (Q^{k+1})^\vee$;
    \item  For any $V\subset Q$, ${}^*(\Phi(V))=(\Phi^\vee)^{-1}({}^*V)$.
\end{itemize}

Therefore, in order for the reduction of $\gamma_W$ to be $[(k+1)^{l}, k^{m-l}]$ and for the induced pairing on $\mathcal{E}_W$ to be non-degenerate, it suffices to find a subspace ${}^*W$ in $Q^\vee=(Q^{k+1})^\vee\oplus (Q^{k})^\vee$ such that 
\begin{itemize}
    \item [($\dagger1$)] ${}^*W$ is maximal isotropic under the pairing $g_Q^{-1}$;
    \item [($\dagger2$)] $(Q^{k})^\vee\cap ({}^*W+(Q^{k+1})^\vee)\subset \Phi^\vee({}^*W)$;
    \item [($\dagger3$)] $\Phi^\vee({}^*W)\subset (Q^{k})^\vee+{}^*W\cap (Q^{k+1})^\vee$.
\end{itemize}

Recall that $Q^\vee$ admits a basis
\[
    Q^\vee = \left\langle e_{\alpha_i}(0) \right\rangle_{1 \le i \le m },
\]
with respect to which the action of $\Phi^\vee$ is given by the matrix $\operatorname{R}\bigl(\operatorname{gr}_P(f)\bigr)$.

We define the following subspaces of $Q^\vee$:
\[
    Q^\vee_{div}=\left \langle e_{\alpha_i}(0)  \right \rangle_{m \mid il},
\quad \text{and} \quad
    Q^\vee_{>,ndiv}=\left \langle e_{\alpha_i}(0)  \right \rangle_{m \nmid il,\  i>\frac{m}{2}}.
\]

\begin{lemma}\label{Q div}
The following statements hold:
\begin{itemize}
    \item [(a)] $\dim Q^\vee_{div}=\operatorname{gcd}(m, l)$, sharing same parity as $m$. The restriction of $g_Q^{-1}$ on $Q^\vee_{div}$ has rank at least $\operatorname{gcd}(m, l)-1$;
    \item [(b)] The pairing between $Q^\vee_{div}$ and $Q^\vee_{>,ndiv}$ vanishes and $Q^\vee_{>,ndiv}$ is isotropic; 
    \item [(c)] For $e_{\alpha_i}(0)\in Q^\vee_{div}$, if $i\leq \lfloor\frac{m}{2}\rfloor$, then $e_{\alpha_i}(0)\in (Q^{k+1})^\vee$ and $e_{\alpha_{i+1}}(0)\in (Q^{k})^\vee$, if $i> \lfloor\frac{m}{2}\rfloor$, then $e_{\alpha_i}(0)\in (Q^{k})^\vee$, and if $\lfloor\frac{m}{2}\rfloor< i < m$, then $e_{\alpha_{i+1}}(0)\in (Q^{k+1})^\vee$;
    \item [(d)] $(Q^\vee_{div}+Q^\vee_{>,ndiv})\cap (Q^k)^\vee\subset \Phi^\vee(Q^\vee_{>,ndiv})$;
    \item [(e)] $\overline{(\Phi^\vee(Q^\vee_{div}+Q^\vee_{>,ndiv}))_{=k+1}}\subset Q^\vee_{>,ndiv}\cap (Q^{k+1})^\vee$, here we use $\overline{(\Phi^\vee(Q^\vee_{div}+Q^\vee_{>,ndiv}))_{=k+1}}$ to denote the projection of $\Phi^\vee(Q^\vee_{div}+Q^\vee_{>,ndiv})$ to $(Q^{k+1})^\vee$.
\end{itemize}
\end{lemma}

\begin{proof}
By the definition of $Q^{\vee}_{div}$, $\dim Q^\vee_{div}=\operatorname{gcd}(m, l)$.  If $m\equiv 0$ then $l\equiv 0$, so $\operatorname{gcd}(m, l)\equiv m$. Notice that when $i\neq m$, if $e_{\alpha_i}(0)\in Q^\vee_{div}$ then $e_{\alpha_{m-i}}(0)\in Q^\vee_{div}$. Since $\alpha_i+\alpha_{m-i}=2n$ when $m\mid il$, so $g_Q^{-1}(e_{\alpha_{j_1}}(0), e_{\alpha_{j_2}}(0))=\pm 1$, hence (a) follows.
    
For (b), we compute $g_Q^{-1}(e_{\alpha_{j_1}}(0), e_{\alpha_{j_2}}(0))$ for $\alpha_{j_1}=j_1k+\lceil \frac{j_1 l}{m}\rceil$, $\alpha_{j_2}=j_2k+\lceil \frac{j_2 l}{m}\rceil$ with $\frac{m}{2}\leq j_1, j_2 \leq m$ and not all of $\frac{j_1l}{m}, \frac{j_2l}{m}$ are integers. By the definition of $g_Q^{-1}$, it equals
\[
    \left(g(e_{\alpha_{j_1}}, e_{\alpha_{j_2}})\right)_{1}. 
\]
By Lemma~\ref{order of g_0}, since $\alpha_{j_1}+\alpha_{j_2}\geq 2n+1$, the order of $g(e_{\frac{m-l}{2}k+i(k+1)}, e_{\frac{m-l}{2}k+j(k+1)})$ is greater or equal to the order of $a_{\alpha_{j_1}+\alpha_{j_2}-2n}/t^{m-1-p_{\alpha_{j_1}}-p_{\alpha_{j_2}}}$. Now we have
\begin{align*}
     \ord_ta_{\alpha_{j_1}+\alpha_{j_2}-2n} &\geq \left\lceil \frac{m(\alpha_{j_1}+\alpha_{j_2}-2n)}{mk+l} \right\rceil \\
     & =\left\lceil j_1+j_2-m+ \frac{m \left\lceil \frac{j_1 l}{m} \right\rceil-j_1l+m \left\lceil \frac{j_2 l}{m} \right\rceil-j_2l}{mk+l} \right\rceil.
\end{align*}
Notice that if $m\nmid jl$ for $j=j_1$ or $j_2$, then $\ord_ta_{\alpha_{j_1}+\alpha_{j_2}-2n}\geq j_1+j_2-m+1$, thus $g_Q^{-1}(e_{\alpha_{j_1}}(0), e_{\alpha_{j_2}}(0))=0$ since $p_{\alpha_j}=m-j$. If we require $j_1\leq \frac{m}{2}$ with $m\mid j_1l$ and $j_2>\frac{m}{2}$ with $m\nmid j_2l$ then the same result follows.
    
For (c), when $i\leq \left\lfloor\frac{m}{2} \right\rfloor$ and $m\mid il$, it suffices to compute that $\alpha_{i}-\alpha_{i-1}=k+1$ and $\alpha_{i+1}-\alpha_{i}=k$. When $\left\lfloor\frac{m}{2}\right\rfloor< i<m$ and $m\mid il$, it suffices to compute that $\alpha_{i}-\alpha_{i-1}=k$ and $\alpha_{i+1}-\alpha_{i}=k+1$.

For (d), we first consider the elements in $Q^\vee_{div}$. By (c) we only need to consider $e_{\alpha_i}$ so that $i> \left\lfloor\frac{m}{2}\right\rfloor$. Now, if $e_{\alpha_i}(0)\in Q^\vee_{div}$ with $i> \left\lfloor\frac{m}{2}\right\rfloor$ and $m\mid il$. Notice that this implies that $m\nmid (i-1)l$ and $i-1>\left\lfloor\frac{m}{2}\right\rfloor$, then $e_{\alpha_{i-1}}\in Q^\vee_{>,ndiv}$ and hence $e_{\alpha_i}(0)=\Phi^\vee(e_{\alpha_{i-1}})\in \Phi^\vee(Q^\vee_{>,ndiv})$. For any $e_{\alpha_i}(0)\in Q^\vee_{>,ndiv}\cap (Q^k)^\vee$, we see that $i>\left\lfloor\frac{m}{2}\right\rfloor+1$ since $\alpha_{\left\lfloor\frac{m}{2}\right\rfloor+1}-\alpha_{\left\lfloor\frac{m}{2}\right\rfloor}=k+1$. Then $e_{\alpha_{i-1}}(0)\in Q^\vee_{>,ndiv}$ otherwise $m\mid (i-l)l$ which implies that $e_{\alpha_i}\in (Q^{k+1})^\vee$ by (c). Hence $e_{\alpha_i}(0)=\Phi^\vee(e_{\alpha_{i-1}})\in \Phi^\vee(Q^\vee_{>,ndiv})$.

For (e), take $e_{\alpha_i}\in Q^\vee_{div}$, by (c), if $i\leq \left\lfloor\frac{m}{2}\right\rfloor$, then $e_{\alpha_{i+1}}(0)\in (Q^{k})^\vee$, and if $ \left\lfloor\frac{m}{2}\right\rfloor<i < m$, then $e_{\alpha_{i+1}}(0)\in (Q^{k+1})^\vee\cap Q^\vee_{>,ndiv}$ since $m\nmid (i+1)l$. For $\Phi^\vee(e_{\alpha_{m}}(0))$, we claim that the projection of $\Phi^\vee(e_{\alpha_{m}}(0))$ to $\left\langle e_{\alpha_{i}} \mid i\leq \left\lfloor\frac{m}{2}\right\rfloor ,e_{\alpha_{i}}\in (Q^{k+1})^\vee \right\rangle$ is zero. So we have 
    \[
    \overline{(\Phi^\vee(Q^\vee_{div}))_{=k+1}}\subset Q^\vee_{>,ndiv}\cap (Q^{k+1})^\vee.
    \]
    For $e_{\alpha_i}(0)\in Q^\vee_{>,ndiv}$, either $\Phi^\vee(e_{\alpha_i}(0))\in (Q^{k})^\vee$ or $\Phi^\vee(e_{\alpha_i}(0))\in (Q^{k+1})^\vee$ which means $\Phi^\vee(e_{\alpha_i}(0))\in Q^\vee_{>,ndiv}$ by the second part of (c).

    Now we need to prove the claim, recall that the action of $\Phi^\vee$ is given by the matrix $\operatorname{R}(\operatorname{gr}_P(f))$, so we need to compute $(a_{\alpha_{m+1-i}})_{m+1-i}$ for $1\leq i \leq \left\lfloor\frac{m}{2}\right\rfloor$ so that $e_{\alpha_i}(0)\in (Q^{k+1})^\vee$. Recall that $e_{\alpha_i}(0)\in (Q^{k+1})^\vee$ means $\alpha_i-\alpha_{i-1}=k+1$, which says that $\left\lfloor \frac{il}{m}\right\rfloor-\left\lfloor\frac{(i-1)l}{m}\right\rfloor=1$. Hence, in particular, $m\nmid (i-1)l$.

    To compute $\ord_t a_{\alpha_{m+1-i}}$, it suffices to compute 
    \[
    m\alpha_{m+1-i}=(m+1-i)mk+m\left \lceil \frac{(m+1-i)l}{m}\right \rceil
    \]
    since $m\nmid (i-1)l$, we see that $\ord_t a_{\alpha_{m+1-i}}> m+1-i$ hence the claim follows.
\end{proof}

\begin{lemma}\label{choice of isotropic subspace of Q div}
    Recall $\dim Q^\vee_{div} = \gcd(m,l) \equiv m$, then we have
    \begin{itemize}
        \item[(i)] If $m\equiv 0$, then there exists an isotropic subspace $(Q^\vee_{div})_0$ of dimension $\frac{\gcd(m,l)}{2}$. 
        \item[(ii)] If $m\equiv 1$, then there exists a vector $v$ of $Q^\vee_{div}$, and an isotropic subspace $(Q^\vee_{div})_0$ such that 
        \[
            g_Q^{-1}(v,v)=1,
            \quad \text{and} \quad
            g_Q^{-1}(v,w)=0 \ \text{ for all } w \in (Q^\vee_{div})_0.
        \]
        Moreover, $v\notin (Q^{k+1})^\vee$.
    \end{itemize}
\end{lemma}

\begin{proof}
    When $m\equiv 0$, then $\dim Q^\vee_{div}\equiv 0$ so the statement follows from (a) of Lemma \ref{Q div}.

    When $m\equiv 1$, then if $\dim Q^\vee_{div}>1$, we can easily pick the vector $v_0$ and the subspace $(Q^\vee_{div})_0$ satisfying the conditions. When $\dim Q^\vee_{div}=\operatorname{gcd}(m,l)=1$, we see that $\gcd(m, 2n)=1$, which implies that $\ord_t a_{i}+\ord_t a_{2n-i}>m$ for any $1\leq i < 2n$. Thus by the similar recurrence relation as in Lemma \ref{pairing g in skew case}, we have $\ord_tg(e_{2n}, e_{2n})=1$ and hence the restriction of $g_Q^{-1}$ on $Q^\vee_{div}=\left \langle e_{2n}(0) \right \rangle$ is non-degenerate, hence the result follows.

    By (c) of Lemma \ref{Q div}, we see that if $v\in (Q^{k+1})^\vee$, then $v$ is a linear combination of $e_{\alpha_i}(0)$ so that $i\leq \lfloor\frac{m}{2}\rfloor$. So any $i, j\leq \lfloor\frac{m}{2}\rfloor=\frac{m-1}{2}$, we see that $\alpha_i+\alpha_j<2n$ so the pairing of $e_{\alpha_i}(0)$ and $e_{\alpha_j}(0)$ vanishes hence $v$ would be isotropic, which is a contradiction.
\end{proof}

\begin{proposition}\label{maximal isotropic in even case}
    Suppose that $m\equiv 0$ and $l\neq 0$. For any maximal isotropic subspace $(Q^\vee_{div})_0$ of $Q^\vee_{div}$, the subspace 
    \[
        {}^*W = (Q^\vee_{div})_0+Q^\vee_{>,ndiv}
    \]
    satisfies conditions ($\dagger1$), ($\dagger2$) and ($\dagger3$) above. Consequently, the reduction of $\gamma_W$ has partition $[(k+1)^{l}, k^{m-l}]$ and the induced pairing on $\sE_W$ is non-degenerate.
\end{proposition}

\begin{proof}
    By (a) and (b) of Lemma \ref{Q div} and Lemma \ref{choice of isotropic subspace of Q div} above, we see that ${}^*W$ is maximal isotropic. 

    Then we need to show that ${}^*W$ satisfies conditions (2) and (3). Notice that $(Q^{k})^\vee\cap ({}^*W+(Q^{k+1})^\vee)$ equals to the projection of ${}^*W$ to $(Q^{k})^\vee$. We may prove that the projection of $Q^\vee_{div}+Q^\vee_{>,ndiv}$ to $(Q^k)^\vee$, which equals to $(Q^\vee_{div}+Q^\vee_{>,ndiv})\cap(Q^k)^\vee$, is contained in $\Phi^\vee(Q^\vee_{>,ndiv})$, which is further contained in $\Phi^\vee({}^*W)$. So this follows from (d) of Lemma \ref{Q div}.

    For (3) we show that 
    \[
    \Phi^\vee(Q^\vee_{div}+Q^\vee_{>,ndiv})\subset (Q^{k})^\vee+Q^\vee_{>,ndiv}\cap (Q^{k+1})^\vee
    \]
    which is equivalent to prove that $\overline{(\Phi^\vee(Q^\vee_{div}+Q^\vee_{>,ndiv}))_{=k+1}}\subset Q^\vee_{>,ndiv}\cap (Q^{k+1})^\vee$, so this follows from (e) of Lemma \ref{Q div}.
\end{proof}

\subsubsection{The case that $m>2n$} If $m>2n$, then we write $m=2n\alpha+s$ with $0\leq s <2n$ and $\alpha>0$. We then consider a new polynomial 
\[
    f^\prime(\lambda) = \lambda^{2n} + a_2^\prime \lambda^{2n - 2} + \cdots + a_{2n - 2}^\prime \lambda^2 + a_{2n}^\prime,
\]
where $a_{2i}^\prime=a_{2i}/t^{2i\alpha}$. 

Then if $s\equiv 0$, we may take $\gamma_0^\prime\in L^+\mathfrak{so}_{2n}$ so that $\chi(\gamma^\prime)=f^\prime(\lambda)$, then $\gamma:=t^\alpha\gamma_0^\prime\in L^+\mathfrak{so}_{2n}$ and $\chi(\gamma)=f(\lambda)$. Thus the reduction of $\gamma$ is $[1^{2n}]$.

If $s\equiv 1$, then we may firstly take $\gamma_0^\prime$ as above, then take $g_0^\prime$ as before associated to $f^\prime(\lambda)$. Then we modify $(\gamma_0^\prime, g_0^\prime)$ to $(\gamma^\prime, g^\prime)$ as before so that $g^\prime$ is optimal. Now we consider $(\gamma^\prime:=t^\alpha\gamma_0^\prime, g^\prime)$, since $\alpha\geq 1$, we may take $P^\prime$ so that $g:=(P^\prime)^\vee g^\prime P^\prime$ is optimal w.r.t. $\{1\}$ and the first column of $\gamma:=(P^\prime)^{-1} \gamma^\prime P^\prime$ is divisible by $t^2$.

\subsection{Modification of Pairings}\label{subsec:modify pairing}

In the previous subsection, we constructed symmetric pairings $g_i$ on each $\cE_i:=\tfrac{\cO[\lambda]}{(f_i(\lambda))}$. Some of these pairings are non-degenerate, while others are degenerate. In this subsection, we modify the pairings inductively according to the modification of the partition $\mathbf{d}_A$. In this way, we eventually obtain a non-degenerate pairing.

\vspace{0.3cm}
\paragraph{\textbf{Type D Case:}}
\begin{itemize}
    \item If $m_1 \equiv 0$: $\widetilde{\bf{d}}_1$ is defined in (D1.1) and (D1.2) in \S~\ref{subsec:modify partition}. By Proposition~\ref{m even and l=0} and~\ref{maximal isotropic in even case}, we have non-degenerate pairing $g_1$. Then, there is no need to modify.
    
    \item If $m_1 \equiv 1$: Let $r \geq 2$ be the smallest integer such that $m_2 \equiv \cdots \equiv m_{r-1} \equiv 0$ but $m_r \equiv 1$.  Recall
    \begin{align}\label{modify partition}
    \begin{split}
        &\left( \widetilde{\bf{d}}_1, \bf{d}_2, \ldots, \bf{d}_{r-1}, \widetilde{\bf{d}}_r \right) \\ 
        &= 
        \begin{cases}
        \left( (k_1+1)^{l_1+1}, k_1^{m_1-l_1-1},\bf{d}_2, \ldots, \bf{d}_{r-1}, (k_r+1)^{l_r-1}, k_r^{m_r-l_r+1} \right), & l_r \neq 0;\\
        \left( (k_1+1)^{l_1+1}, k_1^{m_1-l_1-1},\bf{d}_2, \ldots, \bf{d}_{r-1},  k_r^{m_r-l_r-1}, k_r-1 \right), & l_r = 0,
         \end{cases}
     \end{split}
    \end{align}
    is defined in (D2) of ($\star$) in \S\ref{modify partition}. In the following, we construct a non-degenerate pairing on $\bigoplus_{i=1}^r\cE_i$ as in Proposition~\ref{choice of W}.
\end{itemize}
By induction, we arrive at the following Proposition.

\begin{proposition}\label{prop:D}
    For any $\gamma\in L^\heartsuit \mathfrak{so}_{2n}$, if $\chi_{\gamma}(\lambda)$ decomposes as in \eqref{eq:chi decomposition}, then $\gamma$ has a reduction with partition $\bf{d}_D$ as in ($\star$) in \S\ref{modify partition}.
\end{proposition}
\paragraph{\textbf{Type B Case:}}
Recall that
\[
\mathbf d_A=\mathrm{Re}(\mathbf d_1,\mathbf d_2,\ldots)\cup[1].
\]
Assume that $\ord_t\chi_\gamma(0)=m$.

\begin{itemize}
    \item If $m\equiv 0$: define $\mathbf d_B := [\mathbf d_D,1]$ as in (B1) of $(\star\star)$ in \S\ref{modify partition}.
    \item If $m\equiv 1$: we are in (B2) of $(\star\star)$ in \S\ref{modify partition}. 
    There exists an integer $r>1$ such that $m_r\equiv 1$ and $m_j\equiv 0$ for all $j>r$. 
    The initial segment $(\widetilde{\mathbf d}_1,\ldots,\widetilde{\mathbf d}_{r-1})$ is modified exactly as in the type~D case, and the remaining part
    \[
        (\mathbf d_r,\mathbf d_{r+1},\ldots)\cup[1]
    \]
    is modified to
    \[
        [\widetilde{\mathbf d}_r,\mathbf d_{r+1},\ldots],
        \qquad
        \widetilde{\mathbf d}_r=\bigl[(k_r+1)^{l_r+1},\,k_r^{m_r-l_r-1}\bigr].
    \]
    After reordering, we obtain a type B partition $\bf{d}_B$.
\end{itemize}

\begin{proposition}\label{prop:B}
For any $\gamma\in L^{\heartsuit}\mathfrak{so}_{2n+1}$, if $\chi_\gamma(\lambda)$ decomposes as in \eqref{eq:chi decomposition}, then $\gamma$ admits a reduction whose associated partition is $\mathbf d_B$ (as in $(\star\star)$ of \S\ref{modify partition}).
\end{proposition}

\begin{proof}
    If we are in the case (B1) of ($\star\star$) in \S\ref{modify partition}, then we first consider  $(\gamma_D, g_D)$, so that $\chi_{\gamma_D}(\lambda)=\chi_{\gamma}(\lambda)$, and the reduction of $\gamma_D$ is $\bf{d}_D$. Then we consider $\gamma=\gamma_D\oplus 0$, and $g=g_D\oplus 1$, so $\gamma\in L^\heartsuit \mathfrak{so}_{2n+1}$ with the reduction being $\bf{d}_B$.

    If we are in the case (B2) of ($\star\star$) in \S\ref{modify partition}, we adjoin a rank-one summand $\sE'=\sO$, and $\gamma'=0$ action on $\sE'$. Moreover, we define a pairing $g_n$ on $\sE'$ by $g'(x,y)=txy$. Then the argument in the proof of Proposition~\ref{choice of W} applies verbatim, and yields an element $\gamma\in L^{\heartsuit}\mathfrak{so}_{2n+1}$ whose reduction has partition $\mathbf d_B$.
\end{proof}

We now focus on the second case of type D.   

By the discussion above, each $\gamma_i$ acts on a free $\mathcal{O}$-module $\mathcal{E}_i$ such that $\gamma_i(0)$ has partition $[(k_i+1)^{l_i}, k_i^{m_i-l_i}]$. For $2\leq i \leq r-1$, since $m_i\equiv 0$, so the pairing $g_i$ we constructed before, if already non-degenerate and the parition remains to be $[(k_i+1)^{l_i}, k_i^{m_i-l_i}]$, then we may deal with such $\sE_i$ separately. Such case happens when $l_i\neq 0$ or $l_i=0$ but $k_i\equiv 1$. Then we may assume that $\gamma_i(0)$ has partition $[k_i^{m_i}]$ for $2\leq i \leq r-1$.

Each $\mathcal{E}_i$ carries a symmetric pairing $g_i$ fitting into a short exact sequence
\[
    0\longrightarrow \mathcal{E}_i\stackrel{g_i}{\longrightarrow}\mathcal{E}_i^\vee\longrightarrow Q_i \longrightarrow 0
\]
where $Q_i$ is equipped with an induced non-degenerate pairing $g_{Q_i}$ and a linear operator $\Phi_i$. Now we consider $\oplus_{i=1}^r\sE_i$, then we have 
\[
    0\longrightarrow \oplus_{i=1}^r\sE_i\stackrel{\oplus_{i=1}^rg_i}{\longrightarrow}\oplus_{i=1}^r\sE_i^\vee\longrightarrow \oplus_{i=1}^rQ_i \longrightarrow 0.
\]
Our goal is to find a maximal isotropic subspace $W\subset \oplus_{i=1}^rQ_i$ such that, for its preimage $\mathcal{E}_W \subset \bigoplus_{i=1}^r \mathcal{E}_i^\vee$, the induced endomorphism $\gamma_W$ has the desired partition.

Recall that for $i=1, r$ we have decompositions 
\[
    Q_1=Q^{k_1+1}\oplus Q^{k_1}\quad \text{and} \quad Q_r=Q^{k_r+1}\oplus Q^{k_r}. 
\]
Moreover, for each $1\leq i \leq r$ we have a linear operator $\Phi_i$ on $Q_i$. Although it may happen that $k_1 = k_r + 1$, we always regard $Q^{k_1}$ and $Q^{k_r+1}$ as distinct summands.

We set
\[
    \Phi := \bigoplus_{i=1}^r \Phi_i,
    \qquad
    \Gamma^\vee := \bigoplus_{i=1}^r \gamma_i^\vee,
    \qquad
    \widetilde{\Gamma^\vee} := \bigoplus_{i=1}^r \widetilde{\gamma_i^\vee}.
\]

To unify the notation, for $2 \le j \le r-1$ we set $Q^{k_j} := Q_j$. Thus, we obtain a decomposition 
\[
    Q=Q^{k_1+1}\oplus Q^{k_1}\oplus Q^{k_2}\oplus \cdots \oplus Q^{k_{r-1}}\oplus Q^{k_{r}+1}\oplus Q^{k_r}.
\]
For each $j \ge 0$, we introduce the following filtrations:
\[
    F_{\ge j} := \displaystyle\bigoplus_{\alpha \ge j} Q^\alpha \quad \text{and} \quad F^{\le j} := \displaystyle\bigoplus_{\alpha \le j} Q^\alpha,
\]
and set
\begin{itemize}
    \item $F_{=j} := F_{\ge j} \cap F^{\le j}$;
    \item for a subspace $V \subset Q$, $\overline{(V)_{\ge j}}$ denotes its projection to $F_{\ge j}$.
\end{itemize}

Following the same strategy as before, if $m_r\leq 2n_r$, then for any $j\geq 1$, we compute 
\[
    \operatorname{rk}(\gamma_W)^j = \operatorname{rk}(\Gamma^\vee)^j +\dim W-\dim\left(W\cap \Ker (\Gamma^\vee)^{j-1}\right) -\dim \left(W\cap \left(\tilde{\Gamma^\vee}(\Im(\Gamma^\vee)^{j}\oplus (\Gamma^\vee)^{j-1}W) \right)\right).
\]

If $m_r> 2n_r$, in this case we see that $\dim Q_r=1$ and $\Phi_r=0$. Notice that in this case $\tilde{\Gamma^\vee}$ is no longer invertible, then following the proof of Lemma 2.40 of \cite{WWW24}, we have 
\begin{align*}
    \operatorname{rk}(\gamma_W)^j =& \operatorname{rk}(\Gamma^\vee)^j +\dim W-\dim\left(W\cap \Ker (\Gamma^\vee)^{j-1}\right) -\dim \left(W\cap \left(\tilde{\Gamma^\vee}(\Im(\Gamma^\vee)^{j}\oplus (\Gamma^\vee)^{j-1}W) \right)\right)  \\
    &-\dim\left( \Im(\Gamma^\vee)^{j-1}P_W\cap \Ker \tilde{\Gamma^\vee} \right).
\end{align*}
Notice that by the construction of $\gamma_r$, we see that if $j\geq 2$, then $\Im(\Gamma^\vee)^{j-1}P_W\cap \Ker \tilde{\Gamma^\vee}=0$ and if the projection of $W$ to $Q_r$ is surjective, then $\Im P_W\cap \Ker \tilde{\Gamma^\vee}=0$. Now we assume that the projection of $W$ to $Q_r$ is surjective and the intersection of $W$ with $Q_r$ is zero. And as we will see, the choice of $W$ as in Proposition \ref{choice of W} satisfies this condition.

Then we have 
\[
    W\cap \Ker (\Gamma^\vee)^{j-1}=W\cap F^{\leq j-1}, 
\]
and
\[
    W\cap\big(\tilde{\Gamma^\vee}(\Im(\Gamma^\vee)^{j}\oplus (\Gamma^\vee)^{j-1}W)\big)=W\cap \big(\Phi(F_{\geq j+1})\oplus \Phi(\overline{(W)_{\geq j}}\cap F_{=j})\big).
\]
If $m_r>2n_r$, then under the assumption above, when $j=1$, we see that $\overline{(W)_{\geq j}}\cap F_{=j}=0$.

\begin{lemma}\label{rank equality}
The partition of $\gamma_W$ is the reordering of \eqref{modify partition} if and only if the following rank conditions hold:
\begin{itemize}
    \item[(i)]
    If $l_r \neq 0$, then
    \[
    \rk(\gamma_W)^j - \rk(\Gamma^\vee)^j =
    \begin{cases}
        1, & k_r+1 \le j \le k_1,\\
        0, & \text{otherwise}.
    \end{cases}
    \]
    
    \item[(ii)]
    If $l_r = 0$, then
    \[
    \rk(\gamma_W)^j - \rk(\Gamma^\vee)^j =
    \begin{cases}
        1, & k_r \le j \le k_1,\\
        0, & \text{otherwise}.
    \end{cases}
    \]
\end{itemize}
\end{lemma}

As before, it is more convenient to work on $Q^\vee$ with respect to the pairing $\oplus_{i=1}^rg_{Q_i}^{-1}$. We first decompose 
\[
    Q^\vee=(Q^{k_1+1})^\vee\oplus (Q^{k_1})^\vee\oplus (Q^{k_2})^\vee\oplus \cdots \oplus (Q^{k_{r-1}})^\vee\oplus (Q^{k_{r}+1})^\vee\oplus (Q^{k_r})^\vee.
\]
For any integer $j$, we introduce filtrations on $Q^\vee$ by 
\begin{itemize}
    \item $G^{\leq j}=\oplus_{\alpha\leq j}(Q^{\alpha})^\vee$ then ${}^*F_{\geq j}=G^{\leq j-1}$;
    \item $G_{\geq j}=\oplus_{\alpha\geq j}(Q^{\alpha})^\vee$ then ${}^*F^{\leq j}=G_{\geq j+1}$;
    \item ${G_{=j}}=G^{\leq j}\cap G_{\geq j}$.
\end{itemize}

It follows that 
\[
    {}^*F_{=j}={}^*F_{\geq j}+ {}^*F^{\leq j}=G^{\leq j-1}+G_{\geq j+1},
\]
and
\[
    {}^*(W\cap \Ker (\Gamma^\vee)^{j-1})={}^*W+{}^*F^{\leq j-1}={}^*W+G_{\geq j}.
\]

Moreover, we compute
\begin{align*}
     &\dim{}^*(W\cap\tilde{\Gamma^\vee}(\Im(\Gamma^\vee)^{j}\oplus (\Gamma^\vee)^{j-1}W))  \\
     =&\dim\left(\Phi^\vee({}^*W)+\left({}^*F_{\geq j+1}\cap ({}^*F_{=j}+{}^*(\overline{(W)_{\geq j}}))\right)\right) \\
    =&\dim\left(\Phi^\vee({}^*W)+\left(G^{\leq j}\cap (G^{\leq j-1}+G_{\geq j+1}+({}^*W_{\geq j})\right)\right) \\
    =&\dim\left(\Phi^\vee({}^*W)+\left((G^{\leq j-1}\oplus G_{=j})\cap (G^{\leq j-1}\oplus(G_{\geq j+1}+({}^*W_{\geq j}))\right)\right) \\
    =&\dim\left(\Phi^\vee({}^*W)+\left(G^{\leq j-1}+ (G_{\geq j+1}+({}^*W_{\geq j}))\cap G_{=j}\right)\right) \\
    =&\dim\left(\Phi^\vee({}^*W)+\left(G^{\leq j-1}+\overline{({}^*W_{\geq j})_{=j}} \right)\right).
\end{align*}
Here we use the fact that $\overline{(W)_{\geq j}}=(W+F^{\leq j-1})\cap F_{\geq j}$ and $\overline{({}^*W_{\geq j})_{=j}}$ denotes the projection of ${}^*W_{\geq j}$ to $G_{=j}$. And we remark that if $m_r>2n_r$, then  under the assumption above, when $j=1$, we have $\overline{(W)_{\geq j}}\cap F_{=j}=0$, so the equality holds.

We thus obtain the following result.

\begin{lemma}\label{dim reduction}
For each $j \ge 1$, the difference $\rk (\gamma_W)^j-\rk (\Gamma^\vee)^j$ is equal to
     \[
        \dim \left( \overline{\left(\Phi^\vee({}^*W)\right)_{\geq j}}+\overline{({}^*W_{\geq j})_{=j}} \right)-\dim {}^*W_{\geq j}.
     \]
\end{lemma}

\begin{proof}
    By the above discussion and the fact that $\dim W=\dim {}^*W$, we have $\rk (\gamma_W)^j-\rk (\Gamma^\vee)^j$ is equal to 
    \[
        \dim\left({}^*W+G_{\geq j} \right) + \dim\left(\Phi^\vee({}^*W)+\left(G^{\leq j-1}+\overline{({}^*W_{\geq j})_{=j}} \right)\right)-\dim{}^*W-\dim Q^\vee.
    \]
    First, 
    \begin{align*}
        \dim\left({}^*W+G_{\geq j} \right)-\dim{}^*W&=\dim{}^*W+\dim G_{\geq j}-\dim {}^*W_{\geq j}-\dim {}^*W \\
        &=\dim G_{\geq j}-\dim {}^*W_{\geq j}.
    \end{align*}
    Next,
    \begin{align*}
        &\dim\left(\Phi^\vee({}^*W)+\left(G^{\leq j-1}+\overline{({}^*W_{\geq j})_{=j}} \right)\right) \\
        =&\dim \left(\Phi^\vee({}^*W)+\overline{({}^*W_{\geq j})_{=j}} \right)+\dim G^{\leq j-1}-\dim \left(\Phi^\vee({}^*W)+\overline{({}^*W_{\geq j})_{=j}} \right)^{\leq j-1} \\
        =&\dim \overline{\left(\Phi^\vee({}^*W)+\overline{({}^*W_{\geq j})_{=j}} \right)_{\geq j}}+\dim G^{\leq j-1} \\
        =&\dim\left( \overline{\left(\Phi^\vee({}^*W)\right)_{\geq j}}+\overline{({}^*W_{\geq j})_{=j}} \right) +\dim G^{\leq j-1}
    \end{align*}
    Notice that $G_{\geq j}\oplus G^{\leq j-1}=Q^\vee$, then $\rk (\gamma_W)^j-\rk (\Gamma^\vee)^j$ equals to 
    \[
        \dim \left( \overline{\left(\Phi^\vee({}^*W)\right)_{\geq j}}+\overline{({}^*W_{\geq j})_{=j}} \right)-\dim {}^*W_{\geq j}.
    \]
\end{proof}
   
We have fixed bases for $(Q^{k_1+1})^\vee$, $(Q^{k_1})^\vee$, $(Q^{k_r+1})^\vee$, and $(Q^{k_r})^\vee$ in the previous discussion.

\medskip
\noindent
\textbf{The endpoints $i=1,r$:}
Assume first that $l_1\neq 0$ and $l_r\neq 0$. If $m_r\leq 2n_r$, by Lemma~\ref{Q div}, we have subspaces
\[
(Q_1^\vee)_{div},\ (Q_1^\vee)_{>,ndiv}\subset Q_1^\vee,
\qquad
(Q_r^\vee)_{div},\ (Q_r^\vee)_{>,ndiv}\subset Q_r^\vee.
\]
Moreover, by Lemma~\ref{choice of isotropic subspace of Q div}, we may choose
\[
v_1^+ \in (Q_1^\vee)_{div},\quad ((Q_1^\vee)_{div})_0\subset (Q_1^\vee)_{div},
\qquad
v_r^- \in (Q_r^\vee)_{div},\quad ((Q_r^\vee)_{div})_0\subset (Q_r^\vee)_{div},
\]
as in that lemma.
We set
\[
(Q_1^\vee)_{\blacktriangleright}:=((Q_1^\vee)_{div})_0+(Q_1^\vee)_{>,ndiv},
\qquad
(Q_r^\vee)_{\blacktriangleright}:=((Q_r^\vee)_{div})_0+(Q_r^\vee)_{>,ndiv}.
\]

If $m_r> 2n_r$, recall that in this case $(Q_r)^\vee$ has dimension $1$, we then take $v_r^{-}$ be a vector so that $g_{Q_r}^{-1}(v_r^{-}, v_r^{-})=1$, and set $(Q_r^\vee)_{\blacktriangleright}=0$.

If $l_1=0$, then $Q_1^\vee=(Q^{k_1})^\vee$ and, by the discussion preceding Proposition~\ref{m even and l=0}, the restriction of $\Phi^\vee$ to $Q_1^\vee$ is self-adjoint with respect to the induced pairing. Hence we may choose
\[
v_1^+=v^+,
\qquad
(Q_1^\vee)_{\blacktriangleright}=(Q_1^\vee)_0
\]
as in Lemma~\ref{odd maximal isotropic}.

If $l_r=0$, similarly we may choose
\[
v_r^-=v^-,
\qquad
(Q_r^\vee)_{\blacktriangleright}=(Q_r^\vee)_0
\]
as in Lemma~\ref{odd maximal isotropic}.

\medskip
\noindent
\textbf{The middle blocks $2\le i\le r-1$:}
For $2\le i\le r-1$, the restriction of $\Phi^\vee$ to $(Q^{k_i})^\vee$ is self-adjoint with respect to the induced pairing. By Lemma~\ref{even maximal isotropic}, we have a subspace $(Q^{k_i})^\vee_0\subset (Q^{k_i})^\vee$ and vectors $v_i^\pm\in (Q^{k_i})^\vee$.

For $2\le i\le r-1$, set
\[
({}^*W)_{2}^{\,i}
:=
\bigoplus_{\alpha=2}^{i} (Q^{k_\alpha})^\vee_0
\;+\;
\left\langle\, v_\alpha^{+} \pm \sqrt{-1}\, v_{\alpha+1}^{-}\ \right\rangle_{2\le \alpha \le i-1}.
\]
(Here the choice of signs is fixed once and for all.)

\begin{proposition}\label{choice of W}
Let
\[
    {}^*W=(Q_1^\vee)_{\blacktriangleright}+({}^*W)_{2}^{r-1}+(Q_r^\vee)_{\blacktriangleright}.
\]
Then the reduction of $\gamma_W$ is the reordering of \eqref{modify partition} and the induced pairing on $\sE_W$ is non-degenerate.
\end{proposition}

\begin{proof}
Firstly, by Lemma~\ref{Q div} and Lemma~\ref{even maximal isotropic} we see that ${}^*W$ is maximal isotropic. So the induced pairing on $\sE_W$ is non-degenerate.

By Lemma \ref{dim reduction}, in order to ensure the reduction of $\gamma_W$ is the reordering of \eqref{modify partition}, we need the following conditions:
\[
\dim \left( \overline{\left(\Phi^\vee({}^*W)\right)_{\geq j}}+\overline{({}^*W_{\geq j})_{=j}} \right)-\dim {}^*W_{\geq j}=1
\]
for $k_r+1\leq j \leq k_1$ when $l_r\neq 0$ or for $k_r\leq j \leq k_1$ when $l_r= 0$. And 
\[
\dim \left( \overline{\left(\Phi^\vee({}^*W)\right)_{\geq j}}+\overline{({}^*W_{\geq j})_{=j}} \right)-\dim {}^*W_{\geq j}=0
\]
for other cases.

For any $j$ such that $k_1\geq k_{i}>j>k_{i+1}> k_r+1$ or $k_{r-1}>j >k_r+1$ we see that 
\begin{align*}
        \overline{\left(\Phi^\vee({}^*W)\right)_{\geq j}}&=\Phi^\vee(Q_1^\vee)_{\blacktriangleright}+\Phi^\vee({}^*W)_{2}^{i-1}+\Phi^\vee(Q^{k_i})^\vee_0 + \left\langle \Phi^\vee(v_i^{+}) \right\rangle \\
        \overline{({}^*W_{\geq j})_{=j}}&=0 \\
        {}^*W_{\geq j} &=(Q_1^\vee)_{\blacktriangleright}+({}^*W)_{2}^{i-1}+(Q^{k_i})^\vee_0.
\end{align*}
Thus the condition is satisfied.

If $j=k_i$ with $k_1>k_i>k_r+1$, then 
\begin{align*}
    \overline{\left(\Phi^\vee({}^*W)\right)_{\geq j}}&=\Phi^\vee(Q_1^\vee)_{\blacktriangleright}+\Phi^\vee({}^*W)_{2}^{i-1}+\Phi^\vee(Q^{k_i})^\vee_0 + \left\langle \Phi^\vee(v_i^{+}) \right\rangle \\
    \overline{({}^*W_{\geq j})_{=j}}&=\left\langle v_i^{-} \right\rangle+(Q^{k_i})^\vee_0 \\
    {}^*W_{\geq j}&=(Q_1^\vee)_{\blacktriangleright}+({}^*W)_{2}^{i-1}+(Q^{k_i})^\vee_0.
\end{align*}
By Lemma \ref{even maximal isotropic}, we see that $\left\langle v_i^{-} \right\rangle+(Q^{k_i})^\vee_0=\Phi^\vee(Q^{k_i})^\vee_0 + \left\langle \Phi^\vee(v_i^{+}) \right\rangle$ hence the condition is satisfied.

When $j=k_1$, we see that 
\begin{align*}
    \overline{\left(\Phi^\vee({}^*W)\right)_{\geq j}}&=\Phi^\vee(Q_1^\vee)_{\blacktriangleright}+\left \langle \Phi^\vee(v_1^{+}) \right \rangle \\
    \overline{({}^*W_{\geq j})_{=j}}&=(Q_1^\vee)_{\blacktriangleright}\cap (Q^{k_1})^\vee \\
    {}^*W_{\geq j}&=(Q_1^\vee)_{\blacktriangleright}.
\end{align*}
So we need $(Q_1^\vee)_{\blacktriangleright}\cap (Q^{k_1})^\vee\subset \Phi^\vee(Q_1^\vee)_{\blacktriangleright}+\left \langle \Phi^\vee(v_1^{+}) \right \rangle$. When $l_1=0$, this hold by Lemma \ref{odd maximal isotropic}. When $l_1\neq 0$, by (d) of Lemma \ref{Q div}, we have
\[
(Q_1^\vee)_{\blacktriangleright}\cap (Q^{k_1})^\vee\subset ((Q_1^\vee)_{div}+(Q_1^\vee)_{>, ndiv})\cap (Q^{k_1})^\vee\subset \Phi^\vee((Q_1^\vee)_{>, ndiv}) \subset \Phi^\vee(Q_1^\vee)_{\blacktriangleright}+\left \langle \Phi^\vee(v_1^{+}) \right \rangle.
\]

When $j=k_1+1$, we see that 
\begin{align*}
    \overline{\left(\Phi^\vee({}^*W)\right)_{\geq j}}&=\overline{\left(\Phi^\vee(Q_1^\vee)_{\blacktriangleright}+\left \langle \Phi^\vee(v_1)\right \rangle\right)_{=k_1+1}}  \\
    \overline{({}^*W_{\geq j})_{=j}}&=(Q_1^\vee)_{\blacktriangleright}\cap (Q^{k_1+1})^\vee \\
    {}^*W_{\geq j}&=(Q_1^\vee)_{\blacktriangleright}\cap (Q^{k_1+1})^\vee.
\end{align*}
When $l_1=0$, then $(Q^{k_1+1})^\vee=0$ and hence all these three spaces are $0$ so the condition is satisfied. When $l_1\neq 0$, by (e) of Lemma \ref{Q div}, this case follows.

When $j=k_r+1$ and $l_r\neq 0$, then by Lemma \ref{choice of isotropic subspace of Q div}, we see that $v_r^-\notin (Q^{k_r+1})^\vee$, hence
\begin{align*}
    \overline{\left(\Phi^\vee({}^*W)\right)_{\geq j}}&= \Phi^\vee \left((Q_1^\vee)_{\blacktriangleright}+({}^*W)_{2}^{r-2}\right)+\Phi^\vee((Q^{k_{r-1}})^\vee_0)+\overline{(\left\langle\Phi^\vee(v_{r-1}^+\pm \sqrt{-1}v_r^-) \right\rangle)_{=k+1}}+\overline{\left(\Phi^\vee(Q_r^\vee)_{\blacktriangleright}\right)_{=k_1+1}}  \\
    \overline{({}^*W_{\geq j})_{=j}}&=(Q_r^\vee)_{\blacktriangleright}\cap (Q^{k_r+1})^\vee \\
    {}^*W_{\geq j}&= (Q_1^\vee)_{\blacktriangleright}+({}^*W)_{2}^{r-2}+(Q^{k_{r-1}})^\vee_0+(Q_r^\vee)_{\blacktriangleright}\cap (Q^{k_r+1})^\vee.
\end{align*}

Now by (e) of Lemma \ref{Q div}, we see that $\overline{\left(\Phi^\vee(Q_r^\vee)_{\blacktriangleright}\right)_{=k_1+1}}\subset (Q_r^\vee)_{\blacktriangleright}\cap (Q^{k_r+1})^\vee$, then 
\[
\overline{\left(\Phi^\vee({}^*W)\right)_{\geq j}}+\overline{({}^*W_{\geq j})_{=j}}=\Phi^\vee \left((Q_1^\vee)_{\blacktriangleright}+({}^*W)_{2}^{r-1}\right) +(Q_r^\vee)_{\blacktriangleright}\cap (Q^{k_r+1})^\vee
\]
hence the condition is satisfied.

When $j=k_r$ and $l_r\neq 0$, then 
\begin{align*}
    \overline{\left(\Phi^\vee({}^*W)\right)_{\geq j}}&= \Phi^\vee({}^*W) \\
    \overline{({}^*W_{\geq j})_{=j}}&=\overline{(\langle v_r^- \rangle+(Q_r^\vee)_{\blacktriangleright})_{=k_r}}\\
    {}^*W_{\geq j}&= {}^*W.
\end{align*}
So this case follows from (d) of Lemma \ref{Q div}.

When $j=k_r+1$ and $l_r=0$, then $(Q^{k_r+1})^\vee=0$, hence
\begin{align*}
    \overline{\left(\Phi^\vee({}^*W)\right)_{\geq j}}&= \Phi^\vee \left((Q_1^\vee)_{\blacktriangleright}+({}^*W)_{2}^{r-2}\right)+\Phi^\vee((Q^{k_{r-1}})^\vee_0)+\left\langle\Phi^\vee(v_{r-1}^+) \right\rangle  \\
    \overline{({}^*W_{\geq j})_{=j}}&=0 \\
    {}^*W_{\geq j}&= (Q_1^\vee)_{\blacktriangleright}+({}^*W)_{2}^{r-2}+(Q^{k_{r-1}})^\vee_0.
\end{align*}
So the condition is satisfied.

When $j=k_r$ and $l_r=0$, then $(Q^{k_r+1})^\vee=0$ and 
\begin{align*}
    \overline{\left(\Phi^\vee({}^*W)\right)_{\geq j}}&= \Phi^\vee({}^*W) \\
    \overline{({}^*W_{\geq j})_{=j}}&=\left \langle v_r^- \right \rangle+(Q_r^\vee)_{\blacktriangleright}\\
    {}^*W_{\geq j}&= {}^*W.
\end{align*}
Recall that in this case, $v_r^-=v^-$, $(Q_r^\vee)_{\blacktriangleright}=(Q_r^\vee)_{0}$ as in the Lemma \ref{odd maximal isotropic}, so by Lemma \ref{odd maximal isotropic}, we see that 
\[
\dim \left(\Phi^\vee({}^*W)+\left \langle v_r^- \right \rangle+(Q_r^\vee)_{\blacktriangleright}\right)-\dim {}^*W=1.
\]

For other choices of $j$, the conditions are satisfied obviously.
\end{proof}

\bibliographystyle{alpha}
\bibliography{ref}
\end{document}